\documentclass[reqno, 12pt]{amsart}
\usepackage{fullpage} % Package to use full page
\usepackage{parskip} % Package to tweak paragraph skipping
\usepackage{tikz} % Package for drawing
\usepackage{amsmath}
\usepackage{amssymb} %Package for some math symbols
\usepackage{amsthm} %Package for Thms environment
\usepackage{amsfonts} %Package for math fonts
\usepackage{mathrsfs} %Package for more math fonts
\usepackage{hyperref}
\hypersetup{colorlinks,linkcolor={blue},citecolor={red}}
\usepackage{tikz-cd}
\usepackage{verbatim}
\newtheorem{thm}{Theorem}[section]
\newtheorem{rmk}[thm]{Remark}
\newtheorem{defi}[thm]{Definition}
\newtheorem{prop}[thm]{Proposition}
\newtheorem{lem}[thm]{Lemma}

\newtheorem{claim}[thm]{Claim}
\newcommand{\beq}{\begin{equation}}
\newcommand{\eeq}{\end{equation}}

\newcommand{\Ric}{\operatorname{Ric}}

\newcommand{\ddbar}{\sqrt{-1} \partial \bar{\partial}}
\newcommand{\omegamod}{\omega_{\beta, \operatorname{mod}}}

\newcommand{\K}{K\"{a}hler}    
\newcommand{\KE}{K\"{a}hler--Einstein}

\usepackage[utf8]{inputenc}

\title{Small angle limits of negatively curved K\"{a}hler--Einstein metrics with crossing edge singularities}
\begin{document}

\author[Yuxiang Ji]{Yuxiang Ji}
\address{Department of Mathematics, University of Maryland, College Park, MD, 20740}
\email{yxji@umd.edu}
\thanks{Research suppoted in part by NSF grant DMS-1906370.}
\maketitle

\begin{abstract}
Let $(X, D)$ be a log smooth log canonical pair such that $K_X+D$ is ample. Extending a theorem of Guenancia and building on his techniques, we show that negatively curved K\"{a}hler--Einstein crossing edge metrics converge to K\"{a}hler--Einstein mixed cusp and edge metrics smoothly away from the divisor when some of the cone angles converge to $0$. We further show that near the divisor such normalized K\"{a}hler--Einstein crossing edge metrics converge to a mixed cylinder and edge metric in the pointed Gromov--Hausdorff sense when some of the cone angles converge to $0$ at (possibly) different speeds.
\end{abstract}
\section{Introduction}
\subsection{The small angle world}
Let $X$ be a compact K\"{a}hler manifold of dimension $n$ and $D\subset X$ be a smooth hypersurface. A K\"{a}hler edge metric on $X$ with angle $2\pi \beta$ ($0<\beta\leq1$) along $D$ is a K\"{a}hler metric on $X\setminus D$ that is quasi-isometric to the model edge metric at $D$:
\[
\omega_{\operatorname{cone}}=\frac{\beta^2 \sqrt{-1} dz_1\wedge d\bar{z}_1}{|z_1|^{2(1-\beta)}}+\sum_{i=2}^n \sqrt{-1}dz_i\wedge d\bar{z}_i,
\]
where $z_1,\dots, z_n$ are holomorphic coordinates and $D$ is locally given by $\{z_1=0\}$. Tian generalized Calabi's conjecture to K\"{a}hler--Einstein edge metrics and studied the applications of negatively curved K\"{a}hler--Einstein edge metric to algebraic geometry by letting the cone angle tend to $2\pi$ \cite{Tian96}. Donaldson proposed using K\"{a}hler edge metrics to study the existence problem of smooth K\"{a}hler--Einstein metrics of positive curvature on $X$ by deforming the cone angle to $2\pi$ \cite{donaldson2012kahler}. Since then much research has gone into understanding the large angle limits (when $\beta\to 1$) of K\"{a}hler (--Einstein) edge metrics in relation to the Yau--Tian--Donaldson conjecture. Cheltsov--Rubinstein initiated the program of studying \KE\;edge metrics in another extreme where the cone angle goes to zero \cite{Cheltsov-Rubinstein}. One topic of their program is to understand the limit, when such exists, of \KE\;edge metrics as the cone angle tends to $0$. This paper is following that program. We prove that on a log smooth log canonical pair $(X, D)$, i.e., $X$ is a compact K\"{a}hler manifold and $D=\sum_{i=1}^r (1-\beta_i) D_i$ is a divisor with simple normal crossing support such that $\beta_i\in [0, 1)$ for all $i$, assuming that $K_X+\sum_{i=1}^r D_i$ is ample, then the negatively curved K\"{a}hler--Einstein crossing edge metrics converge to the K\"{a}hler--Einstein mixed cusp and edge metric when some of the cone angles tend to $0$. We further study the asymptotic behavior of the K\"{a}hler--Einstein crossing edge metrics near the divisor and show the rescaled K\"{a}hler--Einstein crossing edge metrics converge to mixed cylinder and edge metrics on $(\mathbb{C}^*)^m\times \mathbb{C}^{n-m}$ when some of the cone angles tend to $0$. A beautiful theorem of Guenancia related the K\"{a}hler--Einstein edge metric to the K\"{a}hler--Einstein cusp metric in the smooth case \cite{guenancia2020kahler}, which confirmed a conjecture made by Mazzeo \cite{mazzeo1999kahler}. Our paper is a generalization of Guenancia's results to the snc case. An added interesting feature of our work is the possibility that multiple angles converge to zero at (possibly) different rates.
%%%%%%%%%%%%%%%%%%%%%%%%%%%%%%%%%%%%%%%%%%%%%%%%%%%%%
\subsection{Guenancia's convergence result}
Let $\mathbb{D}^*$ be the punctured unit disc in $\mathbb{C}$. The first observation is that
\begin{align*}
    \omega_{\eta, \mathbb{D}^*}:=\frac{\eta^2 \sqrt{-1} dz\wedge d\bar{z}}{|z|^{2(1-\eta)}(1-|z|^{2\eta})^2},\quad z\in \mathbb{D}^*, \eta\in(0, 1),
\end{align*}
is a K\"{a}hler edge metric with cone angle $2\pi \eta$ at $0$ and it has constant Ricci curvature $-2$. When $\eta$ tends to $0$, $\omega_{\eta, \mathbb{D}^*}$ converges pointwise (see \eqref{eq: convergence of edge to cusp in 1d} for the detail) to the following cusp metric (also called a Poincar\'{e} metric) on $\mathbb{D}^*$:
\begin{align}
    \omega_{P, \mathbb{D}^*}:=\frac{\sqrt{-1} dz\wedge d\bar{z}}{|z|^2 \left(\log |z|^2\right)^2}.\label{eq: model cusp metric}
\end{align}
In higher dimensions, we consider the pair $(X, D)$ where $X$ is a compact K\"{a}hler manifold of dimension $n$ and $D$ is a smooth divisor such that $K_X+D$ is ample. By Kobayashi \cite[Theorem 1]{kobayashi1984kahler} or Tian--Yau \cite[Theorem 2.1]{TianYau} with complements by Wu \cite{wu2006higher}, there exists a unique complete K\"{a}hler--Einstein metric $\omega_0$ on $X\setminus D$ with cusp singularity along $D$ such that $\operatorname{Ric} \omega_0=-\omega_0$.
\begin{defi}\label{def: cusp metric}
$\omega_0$ is said to have cusp singularities along $D$ if whenever $D$ is locally given by $\{z_1=0\}$, there exists a constant $C>0$ such that 
\begin{equation*}
    C^{-1}\omega_{\operatorname{cusp}}\leq \omega_0\leq C\omega_{\operatorname{cusp}},
\end{equation*}
where
$\omega_{\operatorname{cusp}}$ is the model cusp metric:
\begin{align*}
    \omega_{\operatorname{cusp}}:=\frac{\sqrt{-1}dz_1\wedge d\bar{z}_1}{|z_1|^2\log^2 |z_1|^2}+\sum_{i=2}^n \sqrt{-1}dz_i\wedge d\bar{z}_i.
\end{align*}
\end{defi}
Since ampleness is an open condition, there exists some $\beta_0$ such that for $0<\beta<\beta_0$, $K_X+(1-\beta)D$ is also ample. Thus, by Campana--Guenancia--P\u aun \cite[Theorem A]{GP} and Jeffres--Mazzeo--Rubinstein \cite[Theorem 2]{JMR}, there exists a unique negatively curved K\"{a}hler--Einstein edge metric $\omega_\beta$ for each such small $\beta\in(0, \beta_0]$. The family of metrics $\{\omega_\beta\}_{0\leq \beta< \beta_0}$ can be seen as currents on $X$ satisfying the twisted \KE \;equation:
\begin{align*}
    \Ric \omega_\beta=-\omega_\beta+(1-\beta)[D],\quad 0\leq \beta<\beta_0.
\end{align*}
As a generalization of the observation discussed in the beginning of this section, Guenancia related these two metrics as follows:
\begin{thm}\textup{\cite[Theorem A and B]{guenancia2020kahler}}\label{G Thm A}
Let $\omega_0$ be defined as in Definition \ref{def: cusp metric}. $\{\omega_\beta\}_{0<\beta<\beta_0}$ converge to $\omega_0$ in both the weak topology of currents and the $C^\infty_{\operatorname{loc}}(X\setminus D)$-topology as $\beta\to 0$. Moreover, for $\beta\in(0, 1/2]$, there exists a constant $C>0$ independent of $\beta$ such that on any coordinate chart $U$ where $D$ is given by $\{z_1=0\}$, the \KE\;edge metric $\omega_\beta$ satisfies 
\begin{align}
    C^{-1}\omega_{\beta, \operatorname{mod}}\leq \omega_\beta\leq C \omega_{\beta, \operatorname{mod}},\label{G laplacian est}
\end{align}
where 
\begin{align*}
    \omega_{\beta, \operatorname{mod}}:=\frac{\beta^2\sqrt{-1}dz_1\wedge d\bar{z}_1}{|z_1|^{2(1-\beta)}(1-|z_1|^{2\beta})^2}+\sum_{i=2}^n \sqrt{-1} dz_i\wedge d\bar{z}_i.
\end{align*}
\end{thm}
 
The relation between the convergence result in Theorem \ref{G Thm A} and \eqref{G laplacian est} is that the weak convergence from $(\omega_{\beta})_{0<\beta<\beta_0}$ to $\omega_0$ can be recovered from \eqref{G laplacian est} by using Lebesgue's Dominated Convergence Theorem. 

As an application of Theorem \ref{G Thm A}, Guenancia studied the asymptotic behavior of $\omega_\beta$ near $D$ as $\beta\to 0$. Fix a point $p\in D$, let $U_\beta$ denote the punctured metric ball $B_{\omega_\beta}(p, 1)$ of radius $1$ centered at $p$ with respect to the metric $\omega_\beta$. Then after renormalization by $\beta^{-2}$, there exists a subsequence of the metric spaces $(U_\beta, \frac{1}{\beta^2}\omega_\beta)$ converging to $(\mathbb{C}^*\times \mathbb{C}^{n-1}, \omega_{\operatorname{cyl}})$ in the pointed Gromov--Hausdorff sense, where $\omega_{\operatorname{cyl}}$ is a so-called cylindrical metric: 
\begin{defi}
Let $\pi: \mathbb{C}^n\to \mathbb{C}^*\times \mathbb{C}^{n-1}$ be the universal cover of $\mathbb{C}^*\times \mathbb{C}^{n-1}$ given by $\pi(z_1\dots, z_n)=(e^{z_1},z_2,\dots, z_n)$. A K\"{a}hler metric $\omega_{\operatorname{cyl}}$ on $\mathbb{C}^*\times \mathbb{C}^{n-1}$ is called cylindrical if $\pi^*\omega$ is isometric to the usual Euclidean metric on $\mathbb{C}^n$ up to a complex linear transformation.
\end{defi}

\begin{thm}\textup{\cite[Theorem C]{guenancia2020kahler}}\label{G Thm C}
Let $(\beta_k)_{k\in \mathbb{N}}$ be a sequence of positive numbers converging to $0$. Then, up to extracting a subsequence, there exists a cylindrical metric $\omega_{\operatorname{cyl}}$ on $\mathbb{C}^*\times \mathbb{C}^{n-1}$ such that the metric spaces $(U_{\beta_k}, \beta_k^{-2}\omega_{\beta_k})$ converge in pointed Gromov-Hausdorff topology to $(\mathbb{C}^*\times \mathbb{C}^{n-1}, \omega_{\operatorname{cyl}})$ when $k$ tends to $+\infty$.
\end{thm}
%%%%%%%%%%%%%%%%%%%%%%%%%%%%%%%%%%%%%%%%%%%%%%%%%%%%%%%%%%%%%%%%%%%%%%%%%%%%
%%%%%%%%%%%%%%%%%%%%%%%%%%%%%%%%%%%%%%%%%%%%%%%%%%%%%%%%%%%%%%%%%%%%%%%%
\subsection{The main results}
A natural problem is to generalize Theorems \ref{G Thm A} and \ref{G Thm C} to the snc case when all or some of the cone angles tend to $0$. This possibility is mentioned in \cite{guenancia2020kahler} but there is no detailed proof given. In this paper, we generalize Theorems \ref{G Thm A} and \ref{G Thm C} to the snc setting.

From now on, let $(X, \omega)$ be an $n$-dimensional K\"{a}hler manifold with a smooth K\"{a}hler metric $\omega$. Fix a divisor $\displaystyle D_\beta:=\sum_{i=1}^r (1-\beta_i)D_i$, where $\beta_i\in(0, 1)$ for $i=1,\dots, r$. Assume each $D_i$ is smooth and irreducible. We further assume $D_\beta$ is a simple normal crossing divisor, i.e., for any $p\in \operatorname{supp}(D_\beta)$ lying in the intersection of exactly $m$ components $D_1, \dots, D_m$, there exists a coordinate chart $(U, \{z_i\}_{i=1}^n)$ containing $p$ such that $D_j|_{U}=\{z_j=0\}$ for $j=1,\dots, m$, $m\leq n$. Suppose $\displaystyle K_X+\sum_{i=1}^r D_i$ is ample. Let $s_i$ denote the defining section of $D_i$ and $h_i=|\cdot|_{h_i}$ be a smooth hermitian metric on $L_{D_i}$, which is the line bundle induced by $D_i$. We normalize $h_i$ such that $\log |s_i|_{h_i}^2+1<0$ for each $i$. Denote \[\beta:=(\beta_1,\dots, \beta_r) \in (0, 1)^r.
\]

The following result is well known (see \cite[\S 4]{rubinstein2014smooth} for a survey) .
\begin{thm}\label{existence of KEE} \textup{\cite{GP, JMR, Mazzeo--Rubinstein}}
{\rm (Solution of the Calabi--Tian conjecture in the negative regime)} There exists a unique \KE\;crossing edge metric with negative curvature, denoted by $\omega_{\phi_\beta}=\omega+\ddbar \phi_\beta$ on $X$ with cone angle $2\pi \beta_i$ along $D_i$ for each $i$. In another word, $\omega_{\phi_\beta}$ satisfies the \KE\;edge equation
\[
\Ric \omega_{\phi_\beta}-[D_\beta]=-\omega_{\phi_\beta}. 
\]
\end{thm}

Analogously to \cite{guenancia2020kahler}, let us introduce a reference metric,
\begin{align}\label{defi of Omegabeta}
    \Omega_\beta:=\omega-\sum_{i=1}^r \sqrt{-1}\partial \bar{\partial} \log \left[\frac{1-|s_i|^{2\beta_i}_{h_i}}{\beta_i} \right]^2.
\end{align}

Our first result is as follows.
\begin{thm}\label{Thm one}
Let $\omega_{\phi_\beta}$ be given by Theorem \ref{existence of KEE}. Let $\Omega_\beta$ be given by \eqref{defi of Omegabeta}. There exists a uniform constant $C>0$, independent of $\beta\in(0, \frac{1}{2}]^r$, such that 
\begin{align*}
    C^{-1}\Omega_\beta\leq \omega_{\phi_\beta}\leq C\Omega_\beta.
\end{align*}
\end{thm}

The key point of Theorem \ref{Thm one} is that the constant $C$ is uniform with respect to small $\beta_i, i=1,\dots, r$. According to Theorem \ref{Thm one} and Lebesgue's Dominated Convergence Theorem, we obtain the weak convergence from $\omega_{\phi_\beta}$ to the \KE\;mixed cusp and edge metric $\omega_0$ constructed in \cite{Gmix} as some of the cone angles tend to $0$. In particular, when $\beta\to 0\in[0, 1)^r$, the limiting metric of such $\omega_{\phi_\beta}$ is the unique \KE\;cusp metric on $\displaystyle(X, \sum_{i=1}^r D_i)$ constructed in \cite{kobayashi1984kahler, TianYau, wu2006higher}. More precisely, the following result is shown in section \ref{sec: glo convergence of kee}.
\begin{thm}\label{Thm three}
The \KE\;crossing edge metric $\omega_{\phi_\beta}$ converges to a \KE\;mixed cusp and edge metric on $(X, D_\beta)$ globally in a weak sense and locally in a strong sense when some of the cone angles tend to $0$. In particular, $\omega_{\phi_\beta}$ converges to the \KE\;cusp metric on $\displaystyle(X, \sum_{i=1}^r D_i)$ in the above sense when $\beta\to 0\in[0, 1)^r$.
\end{thm}

\begin{rmk}
\textup{In Theorem \ref{Thm three}, we assume $\displaystyle K_X+\sum_{i=1}^r D_i$ to be ample to ensure the existence of a limiting \KE\;metric by the work of Kobayashi \cite{kobayashi1984kahler} and Tian--Yau--Wu \cite{TianYau,wu2006higher}. An interesting open problem is to study the convergence of $\omega_{\phi_\beta}$ when we only assume the ampleness of $K_X+D_\beta$ for $0<\beta_i\ll 1$, $i=1,\dots, r$.}
\end{rmk}

Theorem \ref{Thm one} and Theorem \ref{Thm three} generalize Guenancia's Theorem \ref{G Thm A} from the smooth case to the snc case.

As an application of Theorem \ref{Thm one}, we study the asymptotic behavior of the \KE\;crossing edge metric $\omega_{\phi_\beta}$ near $D_\beta$ when the smallest cone angle approaches $0$, with possibly other cone angles also converging to $0$.

To state the result, without loss of generality, we assume for $\beta=(\beta_1,\dots, \beta_r)$ there holds $\beta_1\leq \beta_2\leq \cdots\leq \beta_r$. Fix a point $p\in D_\beta$. Choose a coordinate chart $(U, \{z_i\}_{i=1}^n)$ containing $p$ such that $D_j|_{U}=\{z_j=0\}$ for $j=1,\dots, m$, $m\leq n$. Consider a small neighborhood $U_\beta$ about $p$ defined by
\begin{align*}
U_\beta:=\left\{z\in (\mathbb{C}^*)^m\times \mathbb{C}^{n-m}: |z_1|<e^{-\frac{1}{2\beta_1}}, |z_j|<\left(\frac{\beta_1}{\beta_j}\right)^{\frac{1}{\beta_j}}, j=2,\dots, m, |z_\ell|<1, \ell=m+1,\dots, n\right\}.
\end{align*}
We show that after normalization by factor $\beta_1^{-2}$, a subsequence of metrics $\omega_{\phi_\beta}$ converges to a mixed cylinder and edge metric on $(\mathbb{C}^*)^m\times \mathbb{C}^{n-m}$ (see Definition \ref{def of mix} for more details) as $\beta_1$ tends to $0$. The limiting metric has cylindrical part along the component $D_1$ where the cone angle $\beta_1$ approaches $0$ while has conical singularities along other components. More precisely, the third result of this paper is as follows:
\begin{thm}\label{Thm two}
Let $\{\beta_{1, k}\}_{k\in\mathbb{N}}$ be a sequence of positive numbers converging to $0$. Assume further that $\{\beta_{i, k}\}_{k\in\mathbb{N}}$ does not converge to $0$ for each $i=2,\dots, r$ and all $\beta_{i, k}\in (0, \frac{1}{2}]$. Let ${\omega}_{\phi_{\beta_k}}$ be the (negatively curved) \KE\;crossing edge metric on $(X, D_k=\sum_{i=1}^r (1-\beta_{i, k})D_i)$. Then there exists a subsequence of the metric spaces $(U_{\beta_k}, \omega_{\phi_{\beta_k}})$ converging in pointed Gromov-Hausdorff topology to $((\mathbb{C}^*)^m\times \mathbb{C}^{n-m}, \omega_\infty)$ where $\omega_\infty$ is a mixed cylinder and edge metric.
\end{thm}

Theorem \ref{Thm two} is a generalization of \cite[Theorem C]{guenancia2020kahler} which shows the convergence of \KE\;edge metrics to a cylindrical metric in the smooth case. Regarding complex dimension $1$, i.e., in the Riemann surface case, but in the positive curvature regime, Rubinstein--Zhang showed that the (American) football equipped with the Ricci soliton metric converges to the cone-cigar soliton on $\mathbb{R}_{+}$ as two cone angles converge to $0$ at a different speed and to a flat cylindrical metric as two cone angles converge to $0$ at a comparable speed \cite[Theorem 1.1-1.3]{RZ}. In \cite{RZ}, the $S^1$-symmetry of the metric plays an important role in the proof. In higher dimensions, we generalize Theorem \ref{Thm two} to allow more than one cone angles to tend to $0$ and study the limit behavior of metrics under this joint degeneration of cone angles. The result is as follows.
\begin{thm}\label{r2 case}
Let $\{\beta_{1, k}\}_{k\in\mathbb{N}}$ be a sequence of positive numbers converging to $0$. Assume further that for any $i\in \{2, \dots, r\}$ such that $\{\beta_{i, k}\}_{k\in\mathbb{N}}$ also converges to $0$, there holds $\lim_{k\to\infty}\frac{\beta_{1, k}}{\beta_{i, k}}\in[0, 1]$ and all $\beta_{i, k}\in(0, \frac{1}{2}]$. Let ${\omega}_{\phi_{\beta_k}}$ be the (negatively curved) \KE\;crossing edge metric on $(X, D_k=\sum_{i=1}^r (1-\beta_{i, k})D_i)$. Then there exists a subsequence of the metric spaces $(U_{\beta_k}, \omega_{\phi_{\beta_k}})$ converging in pointed Gromov-Hausdorff topology to $((\mathbb{C}^*)^m\times \mathbb{C}^{n-m}, \omega_\infty)$, where $\omega_\infty$ is a mixed cylinder and edge metric with cylindrical part along components whose cone angles converge to $0$ and conical part along other components.
\end{thm}
In the language of \cite{RZ}, \cite[Theorem 1.1-1.3]{RZ} completely describe, in a geometric sense, the boundary behavior of the body of ample angles \cite{Yconvex} of the pair $(S^2, N+S)$, where $N$ and $S$ denote the north and south poles of the Riemann sphere respectively. In higher dimensions, given a pair $(X, \widetilde{D}=\sum_{i=1}^r D_i)$, Theorem \ref{r2 case} is far from being a satisfactory description of the boundary of the body of ample angles of $(X, \widetilde{D})$ in the negative curvature regime. Part of the reason is that different subsequences may converge to different mixed cylinder and edge metrics. A complete characterization of the moduli space of such $(X, \widetilde{D})$ endowed with \KE\;crossing edge metrics in the sense of \cite{rubinstein2014smooth} is still open.
%%%%%%%%%%%%%%%%%%%%%%%%%%%%%%%%%%%%%%%%%%%%%%%%%%%%%%%%%%%%%%%%%%%%%%%%%%%%
%%%%%%%%%%%%%%%%%%%%%%%%%%%%%%%%%%%%%%%%%%%%%%%%%%%%%%%%%%%%%%%%%%%%%%%%
\subsection{Main ingredients of the proofs}\label{main ingre}
We first recall the key ingredient in the proof of Theorem \ref{G Thm A} is the  boundedness of the holomorphic bisectional curvature of the model metric $\omega_{\beta, \operatorname{mod}}$, which makes it possible to use the Chern--Lu inequality to obtain the Laplacian estimates, cf. \cite[Theorem 3.2]{guenancia2020kahler}. Therefore, one way to prove corresponding results of Theorem \ref{G Thm A} in the snc setting is to extend \cite[Theorem 3.2]{guenancia2020kahler} to the snc setting, i.e., prove boundedness of holomorphic bisectional curvature of the model metric $\Omega_\beta$ (see \eqref{defi of Omegabeta} for details). However, this will lead to complicated computations. Indeed, the needed boundedness of holomorphic bisectional curvatures in Theorem \ref{G Thm A} shares the same idea with \cite[Proposition A. 1]{JMR}, where an upper bound on the bisectional curvature of a conical reference metric is obtained when the divisor is smooth. \cite[Proposition A. 1]{JMR} can be extended to the snc setting by assuming all the cone angles $\beta_i$ along each components $D_i$ are less than $\frac{1}{2}$ (instead of less than $1$ in the one smooth divisor case), cf. \cite[Theorem 1.2]{LS}. In this paper, however, we will not follow that approach, i.e., we avoid estimating the bound of the holomorphic bisectional curvature of $\Omega_\beta$ and therefore avoid complicated computations.

Instead, the proof of Theorem \ref{Thm one} uses the regularization arguments of Datar--Song \cite{Datar-Song} to reduce the snc case to the smooth case. More precisely, in Section \ref{sec: small angle limits of KEE}, we approximate the \KE\;crossing edge metric $\omega_{\phi_\beta}$ by a sequence of \KE\;edge metrics with cone singularities along a single component of the divisor $D_\beta$. In the mean time, one realizes the reference metric $\Omega_\beta$ as a sum of \K\;edge metrics $\{\Omega_{\beta_i}\}_{i=1,\dots, r}$ (see \eqref{Defi of Omegabetai}) with cone singularities along a smooth component. The comparison result of $\omega_{\phi_\beta}$ with each $\Omega_{\beta_i}$ is obtained by first showing comparisons between approximating metrics and $\Omega_{\beta_i}$ and then taking a limit in an appropriate sense. By adding things up, one finally obtains the Laplacian estimates of $\omega_{\phi_\beta}$ and $\Omega_\beta$, which is essentially the content of Theorem \ref{Thm one}.

An important observation is that the reference metric $\Omega_\beta$ has the property of converging to a \K\;metric with mixed cusp and edge singularities when some of the cone angles tend to $0$. This observation, combined with the content of Theorem \ref{Thm one}, give us the result of Theorem \ref{Thm three} as a corollary. As another consequence of Theorem \ref{Thm one}, Theorem \ref{Thm two} and Theorem \ref{r2 case} treat the limit behavior of the \KE\;crossing edge metric $\omega_{\phi_\beta}$ near the divisor $D_\beta$ when some of the cone angles approach $0$. After fixing a point in the divisor $D_\beta$, we first rescale the reference metric to obtain its convergence to a mixed cylinder and edge metric (see Definition \ref{def of mix}) as the smallest cone angle tends to $0$ in a small neighborhood of $D_\beta$. To obtain the pointed Gromov--Hausdorff convergence of the rescaled \KE\;crossing edge metric $\omega_{\phi_\beta}$ near the divisor, we actually show a stronger local smooth convergence result. We use Theorem \ref{Thm one} and the limit behavior of $\Omega_\beta$ mentioned above to obtain $C^0$-estimates of rescaled $\omega_{\phi_\beta}$. By a standard use of Evans--Krylov theory and Arzel\`{a}--Ascoli Theorem, we obtain the $C^\infty_{\operatorname{loc}}$-convergence of the rescaled $\omega_{\phi_\beta}$ as some of the cone angles tend to $0$.

\smallskip
\textbf{Acknowledgments.}
The author would like to thank H. Guenancia and J. Sturm for helpful conversations. The author is especially grateful to Y. A. Rubinstein for suggesting this problem and for his guidance and encouragement. This research was partly supported by NSF grant DMS-1906370.
\section{Small angle limits of the K\"ahler--Einstein crossing edge metrics}\label{sec: small angle limits of KEE}

Let $\mathbb{D}^*=\{z\in\mathbb{C}: 0<|z|<1\}$ be the punctured unit disc in $\mathbb{C}$. The first observation is, for $\eta\in(0, 1)$, the following K\"{a}hler metric
\begin{align*}
    \omega_{\eta, \mathbb{D}^*}:&=-\sqrt{-1}\partial \bar{\partial}\log(1-|z|^{2\eta})\\
    &=\frac{\sqrt{-1}\eta^2 |z|^{2\eta-2}}{(1-|z|^{2\eta})^2}dz\wedge d\bar{z}
\end{align*}
has negative constant curvature and cone singularity with cone angle $2\pi \eta$ at $0\in\mathbb{C}$. Indeed, direct calculation   using the Poincar\'{e}--Lelong formula \cite{GH} yields
\begin{align}
    \operatorname{Ric} \omega_{\eta, \mathbb{D}^*}=2\pi(1-\eta) \delta_0-2\omega_{\eta, \mathbb{D}^*},\label{eq: ric of omegaetaD}
\end{align}
where $\delta_0$ denotes the Dirac measure at $0$. 

For fixed $z\in \mathbb{D}^*$, 
\begin{equation}\label{eq: convergence of edge to cusp in 1d}
    \lim_{\eta\to 0} \frac{\eta^2 |z|^{2\eta-2}}{(1-|z|^{2\eta})^2}=\frac{1}{|z|^2\left(\log|z|^2\right)^2}.
\end{equation}
Thus, $\omega_{\eta, \mathbb{D}^*}$ converges uniformly to the Poincar\'{e} metric $\omega_{P, \mathbb{D}^*}$ defined in \eqref{eq: model cusp metric} for any compact $K\Subset \mathbb{D}^*$ when $\eta$ tends to $0$.
Note that
\begin{align*}
    \operatorname{Ric}\omega_{P, \mathbb{D}^*}=2\pi \delta_0-2\omega_{P, \mathbb{D}^*}.
\end{align*}
Thus, $\omega_{P, \mathbb{D}^*}$ is a \KE\;cusp metric on $\mathbb{D}^*$ with cusp singularity at $0$. Next we introduce a reference metric that generalizes $\omega_{\eta, \mathbb{D}^*}$ to higher dimensional manifolds.

\subsection{The reference metric}

From now on, let $(X, D_\beta)$ be an $n$-dimensional K\"{a}hler manifold with an $\mathbb{R}$-divisor $D_\beta=\sum_{i=1}^r (1-\beta_i)D_i$ such that $K_X+D_\beta$ is ample, where $\beta_i\in(0, 1)$ for $i=1,\dots, r$. Assume each $D_i$ is smooth and irreducible. We further assume $D_\beta$ is a simple normal crossing divisor, i.e., for any $p\in \operatorname{supp}(D_\beta)$ lying in the intersection of exactly $m$ divisors $D_1, \dots, D_m$, $m\leq n$, there exists a coordinate chart $(U, \{z_i\}_{i=1}^n)$ containing $p$ such that $D_j|_{U}=\{z_j=0\}$ for $j=1,\dots, m$. Let $s_i$ denote a defining holomorphic section of $D_i$ and $h_i=|\cdot|_{h_i}$ be a smooth hermitian metric on $L_{D_i}$, which is the line bundle induced by $D_i$. Let $\theta_i$ denote the curvature form of each $(L_{D_i}, h_i)$. We normalize $h_i$ such that $\log |s_i|_{h_i}^2+1<0$ for each $i$. Let $\omega$ be a fixed smooth \K\;metric with $\displaystyle [\omega]=c_1(K_X+D_\beta)$. Below we denote $\beta:=(\beta_1,\dots, \beta_r)$.

Define the reference metric:
\begin{align}\label{mod metric snc case}
    \Omega_\beta:=\omega-\sum_{i=1}^r \sqrt{-1}\partial \bar{\partial} \log \left[\frac{1-|s_i|^{2\beta_i}_{h_i}}{\beta_i} \right]^2.
\end{align}

\begin{rmk}
\textup{The appearance of $\beta_i$ in the denominator of the log term in the potential function does not affect the definition of the reference metric. We use this convention, following \cite{guenancia2020kahler}, since the potential function in \eqref{mod metric snc case} defined in such a way will be shown to converge weakly to a potential function for some \K\;cusp metric. See Lemma \ref{convergence of the refer metric} for details.}
\end{rmk}

$\Omega_\beta$ can be seen as a generalization of $\omega_{\eta, \mathbb{D}^*}$ to higher dimensiaonl manifolds. First, let us recall \cite[Lemma 3.1]{guenancia2020kahler}.

\begin{lem}\label{shouchong}
$\Omega_\beta$ is a K\"{a}hler edge form  with cone angle $2\pi\beta_i$ along $D_i$ for $i=1,\dots, r$. More precisely, 
\begin{align}\label{formula for omegabeta}
    \Omega_\beta=\omega +2\cdot\sum_{i=1}^r \left(\sqrt{-1}\frac{\beta_i^2}{|s_i|_{h_i}^{2-2\beta_i}(1-|s_i|^{2\beta_i}_{h_i})^2}\langle D^{1, 0}s_i, D^{1, 0}s_i \rangle-\frac{\beta_i |s_i|^{2\beta_i}_{h_i}}{1-|s_i|^{2\beta_i}_{h_i}}\theta_i\right),
\end{align}
where $D^{1, 0}$ is the $(1, 0)$-part of the Chern connection of $(L_{D_i}, h_i)$ for each $i$. Up to rescaling $\{h_i\}_{i=1,\dots, r}$, $\Omega_\beta\geq \frac{1}{2}\omega$.
\end{lem}
\begin{proof}
A concise proof for the case $r=1$ is given in \cite[Lemma 3.1]{guenancia2020kahler}. For the reader's convenience, we give a detailed proof here. In fact, it suffices to show \eqref{formula for omegabeta} when $r=1$. Hence, below we suppose $r=1$ and drop the subscript $i$ for simplicity.

If we set
\begin{align*}
    f(x)&=-\log \left(\frac{1-x^\beta}{\beta} \right)^2,\\
    \phi&=|s|^2_h,
\end{align*}
then $\Omega_\beta=\omega+\ddbar f\circ \phi$.  Recall there holds $\ddbar f\circ \phi=\sqrt{-1}(f^{\prime\prime}(\phi)\partial \phi\wedge \bar{\partial}\phi+f^\prime(\phi)\partial \bar{\partial}\phi)$.
We calculate 
\begin{align*}
    f^\prime&=\frac{2\beta x^{\beta-1}}{1-x^\beta},\\
    f^{\prime\prime}&=\frac{-2\beta x^{\beta-2}}{1-x^\beta}+\frac{2\beta^2 x^{\beta-2}}{(1-x^\beta)^2}.
\end{align*}
Then
\begin{align*}
    \Omega_\beta=\omega + \sqrt{-1}\cdot \frac{2\beta|s|_h^{2\beta-2}}{1-|s|_h^{2\beta}} \partial \bar{\partial} |s|_h^2+\sqrt{-1}\cdot\left(\frac{2\beta^2 |s|_h^{2\beta-4}}{(1-|s|_h^{2\beta})^2}-\frac{2\beta |s|_h^{2\beta-4}}{1-|s|_h^{2\beta}} \right) \partial |s|_h^2\wedge \bar{\partial} |s|_h^2.
\end{align*}
Note 
\begin{align}\label{eq: forms related to s norm}
    \partial \bar{\partial} |s|_h^2&=\bar{s}\partial s\wedge \bar{\partial}h+|s|^2\partial \bar{\partial} h+h\partial s\wedge \bar{\partial}\bar{s}+s\partial h\wedge \bar{\partial}\bar{s},\\
    \partial |s|_h^2\wedge \bar{\partial}|s|^2_h&=|s|^4\partial h\wedge \bar{\partial}h+sh|s|^2 \partial h\wedge \bar{\partial}\bar{s}+\bar{s}h |s|^2\partial s\wedge \bar{\partial}h+|s|^2h^2 \partial s\wedge \bar{\partial} \bar{s},
\end{align}
and the fact
\begin{align}\label{eq: D10s and curv}
    \theta&=-\ddbar \log h=\sqrt{-1}(\frac{\partial h\wedge \bar{\partial}h}{h^2}-\frac{\partial \bar{\partial}h}{h})\\
    \langle D^{1, 0}s, D^{1, 0}s\rangle&=\langle \partial s \cdot e+\frac{\partial h}{h}s\cdot e, \partial s \cdot e+\frac{\partial h}{h}s\cdot e\rangle\\
    &=h\partial s\wedge \bar{\partial}\bar{s}+\bar{s}\partial s\wedge \bar{\partial}h+s\partial h\wedge \bar{\partial}\bar{s}+\frac{|s|^2}{h}\partial h\wedge \bar{\partial}h.
\end{align}

We calculate
\begin{align*}
    \Omega_\beta&=\omega + \sqrt{-1} \frac{2\beta|s|_h^{2\beta-2}}{1-|s|_h^{2\beta}} \partial \bar{\partial} |s|_h^2+\sqrt{-1}\left(\frac{2\beta^2 |s|_h^{2\beta-4}}{(1-|s|_h^{2\beta})^2}-\frac{2\beta |s|_h^{2\beta-4}}{1-|s|_h^{2\beta}} \right) \partial |s|_h^2\wedge \bar{\partial} |s|_h^2\\
    &=\omega+ \left(\sqrt{-1} \frac{2\beta|s|_h^{2\beta-2}}{1-|s|_h^{2\beta}} |s|^2 \partial \bar{\partial}h-\sqrt{-1}\frac{2\beta |s|_h^{2\beta-4}}{1-|s|_h^{2\beta}} |s|^4 \partial h\wedge \bar{\partial} h\right)\\
    &+ \sqrt{-1} \frac{2\beta|s|_h^{2\beta-2}}{1-|s|_h^{2\beta}}(\bar{s}\partial s\wedge \bar{\partial}h+h\partial{s}\wedge \bar{\partial}\bar{s}+s\partial h\wedge \bar{\partial}\bar{s})\\
    &+\sqrt{-1}\left(\frac{2\beta^2 |s|_h^{2\beta-4}}{(1-|s|_h^{2\beta})^2}-\frac{2\beta |s|_h^{2\beta-4}}{1-|s|_h^{2\beta}} \right)(sh|s|^2 \partial h\wedge \bar{\partial}\bar{s}+\bar{s}h |s|^2\partial s\wedge \bar{\partial}h+|s|^2h^2 \partial s\wedge \bar{\partial} \bar{s})\\
    &+\sqrt{-1}\frac{2\beta^2 |s|_h^{2\beta-4}}{(1-|s|_h^{2\beta})^2} |s|^4\partial h\wedge \bar{\partial}h\\
    &=\omega-\frac{2\beta |s|_h^{2\beta}}{1-|s|_h^{2\beta}}\sqrt{-1}(\frac{\partial h\wedge \bar{\partial}h}{h^2}-\frac{\partial \bar{\partial}h}{h})\\
    &+\sqrt{-1}\frac{2\beta^2|s|^{2\beta-2}_h}{(1-|s|_h^{2\beta})^2}(h\partial s\wedge \bar{\partial}\bar{s}+\bar{s}\partial s\wedge \bar{\partial}h+s\partial h\wedge \bar{\partial}\bar{s}+\frac{|s|^2}{h}\partial h\wedge \bar{\partial}h)\\
    &=\omega+2\cdot\sqrt{-1}\frac{\beta^2|s|^{2\beta-2}_h}{(1-|s|_h^{2\beta})^2}\langle D^{1, 0}s, D^{1, 0}s\rangle-2\cdot\frac{\beta|s|_h^{2\beta}}{1-|s|_h^{2\beta}}\theta,
\end{align*}
which is what we need. Since $\displaystyle\sqrt{-1}\frac{\beta_i^2}{|s_i|_{h_i}^{2-2\beta_i}(1-|s_i|^{2\beta_i}_h)^2}\langle D^{1, 0}s_i, D^{1, 0}s_i \rangle$ contributes as a non-negative $(1, 1)$-form for each $i$, we will show that up to rescaling $h_i$, $\displaystyle\frac{\beta_i|s_i|^{2\beta_i}_{h_i}}{1-|s_i|^{2\beta_i}_{h_i}}$ can be made arbitrarily small to conclude that $\Omega_\beta\geq \frac{1}{2}\omega$. To see this, consider the function $\displaystyle f_{\beta_i}(t):=\frac{\beta_i t^{\beta_i}}{1-t^{\beta_i}}$. $f_{\beta_i}(t)$ is increasing in $(0, 1)$ and satisfies $f_{\beta_i}(0)=0$.
Hence for any $\delta>0$, $\exists t_\delta\in(0, 1)$, such that for $t\in (0, t_\delta]$, $f_{\beta_i}(t)\leq \delta$ for each $i=1,\dots, r$. Now take $\displaystyle \delta=\frac{1}{4r\cdot \sup_{X, i} \operatorname{tr}_\omega \theta_i}$ and rescale each $h_i$ such that $|s_i|^2_{h_i}\leq t_\delta$. Then
\begin{align*}
    2\cdot \sum_{i=1}^r -\frac{\beta_i|s_i|^{2\beta_i}_{h_i}}{1-|s_i|^{2\beta_i}_{h_i}}\theta_i\geq -\frac{1}{2}\omega
\end{align*}
and therefore $\Omega_\beta\geq \frac{1}{2}\omega$.
\end{proof}

When $r=1$ and $\beta=\beta_1\in(0, \frac{1}{2}]$, the following result of Guenancia states that the reference metric $\Omega_\beta$ has uniformly bounded holomorphic bisectional curvatures on $X\setminus D_\beta$.
\begin{lem}\textup{\cite[Theorem 3.2]{guenancia2020kahler}}\label{erchong}
When $r=1$, there exists a constant $C>0$ depending only on $X$ such that for all $\beta\in(0, \frac{1}{2}]$, the holomorphic bisectional curvature of $\Omega_\beta$ is bounded by $C$.
\end{lem}

\begin{proof}
For the reader's convenience, we give an alternative proof for the lemma following \cite[Lemma 2.3]{JMR}. Let us assume $r=1$ and drop all the subscripts of $\beta$, $D$, $s$, $h$ and $\theta$. We will show that the curvature tensor of $\Omega_\beta$ is uniformly bounded on $X\setminus D$. Fix a point $p\in X\setminus D$. We can find local holomorphic coordinates such that $s=z_1$ and the hermitian metric $h$ on $L_D$ is given by $h=e^{-\phi}$ while $\phi(p)=0$ and $d\phi(p)=0$. In these local coordinates, write
\begin{align*}
    \omega&=\sqrt{-1}\tilde{g}_{i\bar{j}}dz^i\wedge d\bar{z}^j,\\
    \theta&=\sqrt{-1}\theta_{i\bar{j}}dz^i\wedge d\bar{z}^j,
\end{align*}
where $\tilde{g}_{i\bar{j}}$ and $\theta_{i\bar{j}}$ are smooth functions. Moreover, we have
\begin{align*}
    \langle D^{1, 0}s, D^{1, 0}s\rangle&=\langle dz_1\cdot e-z_1\partial \phi \cdot e, dz_1 \cdot e-z_1\partial \phi \cdot e \rangle\\
    &=e^{-\phi}(1-\bar{z}_1\frac{\partial \phi}{\partial \bar{z}_1}-{z}_1\frac{\partial \phi}{\partial {z}_1}+|z_1|^2\frac{\partial \phi}{\partial z_1}\frac{\partial \phi}{\partial \bar{z}_1}) dz^1\wedge d\bar{z}^1\\
    &+\sum_{i=2}^n e^{-\phi}(-\bar{z}_1\frac{\partial \phi}{\partial \bar{z}_i}+|z_1|^2\frac{\partial \phi}{\partial z_1}\frac{\partial \phi}{\partial \bar{z}_i}) dz^1\wedge d\bar{z}^i\\
    &+\sum_{j=2}^n e^{-\phi}(-{z}_1\frac{\partial \phi}{\partial {z}^j}+|z_1|^2\frac{\partial \phi}{\partial z_j}\frac{\partial \phi}{\partial \bar{z}_1}) dz^j\wedge d\bar{z}^1\\
    &+\sum_{k, \ell=2}^n e^{-\phi} |z_1|^2\frac{\partial \phi}{\partial z_k}\frac{\partial \phi}{\partial \bar{z}_\ell} dz^k\wedge d\bar{z}^\ell.
\end{align*}
Hence, writing $\Omega_\beta=\sqrt{-1}g_{i\bar{j}}dz^i\wedge d\bar{z}^j$ we have
\begin{align*}
    g&=\tilde{g}_{i\bar{j}}-\frac{\beta |s|_h^{2\beta}}{1-|s|_h^{2\beta}}\theta_{i\bar{j}}+\frac{\beta^2 |s|^{2\beta-2}_h}{(1-|s|_h^{2\beta})^2} \langle D^{1, 0}s, D^{1, 0}s \rangle\\
    &=\tilde{g}_{i\bar{j}}-\frac{\beta e^{-\beta \phi} |z_1|^{2\beta}}{1-|z_1|^{2\beta}\cdot e^{-\beta\phi}}\theta_{i\bar{j}}+\frac{\beta^2 e^{-(\beta-1)\phi} |z_1|^{2\beta-2}}{(1-|z_1|^{2\beta}e^{-\beta\phi})^2} \langle D^{1, 0}s, D^{1, 0}s \rangle.
\end{align*}
For each components, we have
\begin{align*}
    g_{1\bar{1}}&=\tilde{g}_{1\bar{1}}-\frac{\beta e^{-\beta \phi} |z_1|^{2\beta}}{1-|z_1|^{2\beta}\cdot e^{-\beta\phi}}\theta_{1\bar{1}}+\frac{\beta^2 e^{-(\beta-1)\phi} |z_1|^{2\beta-2}}{(1-|z_1|^{2\beta}e^{-\beta\phi})^2} e^{-\phi}(1-\bar{z}_1\phi_{\bar{z}_1}-z_1\phi_{z_1}+|z_1|^2\phi_{z_1}\phi_{\bar{z}_1}),\\
    g_{1\bar{j}}&=\tilde{g}_{1\bar{j}}-\frac{\beta e^{-\beta \phi} |z_1|^{2\beta}}{1-|z_1|^{2\beta}\cdot e^{-\beta\phi}}\theta_{1\bar{j}}+\frac{\beta^2 e^{-(\beta-1)\phi} |z_1|^{2\beta-2}}{(1-|z_1|^{2\beta}e^{-\beta\phi})^2} e^{-\phi}(-\bar{z}_1\phi_{\bar{z}_j}+|z_1|^2\phi_{z_1}\phi_{\bar{z}_j}),\;j=2,\dots, n,\\
    g_{k\bar{\ell}}&=\tilde{g}_{k\bar{\ell}}-\frac{\beta e^{-\beta \phi} |z_1|^{2\beta}}{1-|z_1|^{2\beta}\cdot e^{-\beta\phi}}\theta_{k\bar{\ell}}+\frac{\beta^2 e^{-(\beta-1)\phi} |z_1|^{2\beta-2}}{(1-|z_1|^{2\beta}e^{-\beta\phi})^2} e^{-\phi}|z_1|^2\phi_{z_k}\phi_{\bar{z}_\ell},\;k, \ell=2,\dots, n.
\end{align*}
Consider a change of coordinate $\xi=z_1^\beta/\beta$, then we have
\begin{equation*}
    dz_1\wedge d\bar{z}_1=|\beta \xi|^{\frac{2}{\beta}-2}d\xi\wedge d\bar{\xi}.
\end{equation*}
Moreover, in coordinates $(\xi, z_2,\dots, z_n)$, $g_{i\bar{j}}=O(|\xi|^{\frac{2}{\beta}-2})=O(1)$ for $i, j=1,\dots, n$. Since the curvature tensor is given by 
\begin{equation*}
    R_{i\bar{j}k\bar{\ell}}=-g_{i\bar{j}, k\bar{\ell}}+g^{s\bar{t}}g_{i\bar{t}, k}g_{s\bar{j}, \ell},
\end{equation*}
it suffices to show $g_{i\bar{j}, k}$ and $g_{i\bar{j}, k\bar{\ell}}$ are bounded to finish the proof. Indeed, we have 
\begin{align*}
    g_{i\bar{j}, k}&=O(|\xi|^{\frac{2}{\beta}}-3),\\
    g_{i\bar{j}, k\bar{\ell}}&=O(|\xi|^{\frac{2}{\beta}}-4).
\end{align*}
When $0<\beta\leq\frac{1}{2}$, the curvature tensor is bounded.
\end{proof}

\begin{rmk}
\textup{As pointed out to the author by H. Guenancia, J. Sturm's trick (see \cite[p. 62]{rubinstein2014smooth}) can also be applied to simplify the proof of the curvature bounds.}
\end{rmk}

\subsection{A priori estimates}
By Theorem \ref{existence of KEE}, there exists a unique \KE\;crossing edge metric on $X$ with cone angle $2\pi \beta_i$ along each $D_i$, denoted by $\omega_{\phi_\beta}=\omega+\sqrt{-1}\partial \bar{\partial} \phi_\beta$ such that 
\begin{align}\label{MA eq}
    \left\{
    \begin{array}{ll}
     \displaystyle(\omega+\sqrt{-1}\partial \bar{\partial} \phi_\beta)^n    =\frac{e^{f+\phi_\beta} \omega^n}{\Pi_{i=1}^r |s_i|_{h_i}^{2(1-\beta_i)}},  \\
      \displaystyle\omega_{\phi_\beta}=\omega+\sqrt{-1}\partial \bar{\partial} \phi_\beta>0,  
    \end{array}
    \right.
\end{align}
where $f\in C^\infty(X)$. In this section, we establish a Laplacian estimate for $\omega_{\phi_\beta}$ with respect to the reference metric $\Omega_\beta$ by proving the following result.

\begin{thm}\label{compare omegaphi and Omegabeta}
For $\beta=(\beta_1,\dots, \beta_r)\in(0, \frac{1}{2}]^r$, there exists a constant $C>0$ (independent of $\beta_1,\dots, \beta_r$) such that 
\begin{align}\label{bound for kee metric wrt ref metric}
    C^{-1}\Omega_\beta\leq \omega_{\phi_\beta}\leq C\Omega_\beta,
\end{align}
on $X\setminus \operatorname{supp}(D_\beta)$.
\end{thm}

Define 
\begin{align*}
    \psi_{\beta_i}:=-\log\left[ \frac{1-|s_i|_{h_i}^{2\beta_i}}{\beta_i}\right]^2,
\end{align*}
then 
\begin{align*}
    \Omega_\beta=\omega+\sum_{i=1}^r \ddbar \psi_{\beta_i}.
\end{align*}
Consider the \K\;edge metric with cone singularity along a smooth divisor:
\begin{align}\label{Defi of Omegabetai}
    \Omega_{\beta_i}:=\frac{1}{r}\omega+\ddbar \psi_{\beta_i}, \quad i=1,\dots, r.
\end{align}

\begin{prop}\label{comp omegaphi and Omegabetai}
Assume $\beta\in(0, \frac{1}{2}]^r$. There exists a uniform constant $C'>0$ such that for any $i=1,\dots, r$, 
\begin{align}
    \omega_{\phi_\beta}\geq C'\Omega_{\beta_i},\label{omegaphi g Omegai}
\end{align}
on $X\setminus \operatorname{supp}(D_i)$.
\end{prop}

Let us first show Theorem \ref{compare omegaphi and Omegabeta} is an easy consequence of Proposition \ref{comp omegaphi and Omegabetai}.
\begin{proof}[Proof of Theorem \ref{compare omegaphi and Omegabeta}]
Adding \eqref{omegaphi g Omegai} for all $i=1,\dots, r$, there holds
\[
\omega_{\phi_\beta}\geq \frac{C'}{r}\cdot\Omega_{\beta}.
\]
By Lemma \ref{shouchong}, $\Omega_\beta$ has cone singularities with cone angle $2\pi \beta_i$ along $D_i$, for $i=1.\dots, r$. Thus, $\Omega_\beta^n$ and $\omega_{\phi_\beta}^n$ are uniformly equivalent on $X\setminus \operatorname{supp}(D_\beta)$ \cite{Datar-Song}. The lower bound on the metric $\omega_{\phi_\beta}$ then implies the upper bound on $\omega_{\phi_\beta}$ \cite[\S 7.2]{rubinstein2014smooth}, i.e., there exists a uniform constant $C>0$ such that \eqref{bound for kee metric wrt ref metric} holds on $X\setminus \operatorname{supp}(D_\beta)$.
\end{proof}

\begin{proof}[Proof of Proposition \ref{comp omegaphi and Omegabetai}]
We divide the proof into three steps. First, we approximate $\omega_{\phi_\beta}$ by using \KE\;edge metrics with singularities only along a smooth divisor. Then we deduce the zero-order estimates for $\Omega_{\beta_i}$ and the approximating metrics. In the last step, we show the approximating metrics are bounded below by some constant times $\Omega_{\beta_i}$, and hence so is $\omega_{\phi_\beta}$.

\begin{rmk}
\textup{
When $r=1$, the proof of \eqref{omegaphi g Omegai} is already given in \cite[Proposition 4.2]{guenancia2020kahler}. The main difference for the case $r>1$ is that $\omega_{\phi_\beta}$ admits crossing edge singularities. Therefore, we first prove similar estimates as in \eqref{omegaphi g Omegai} for a sequence of approximating metrics following the arguments in \cite{guenancia2020kahler}. Then we show \eqref{omegaphi g Omegai} can be obtained by approximating $\omega_{\phi_\beta}$ using K\"{a}hler edge metrics in a proper way. 
}
\end{rmk}

\textbf{Step 1: Approximating $\boldsymbol{\omega_{\phi_\beta}}$.}

Let $f_i:=\log \left(\Pi_{j=1, j\neq i}^r |s_j|^{2(1-\beta_j)}_{h_j} \right)$ for $i=1,\dots, r$. For some constant $A\gg 1$, $\ddbar f_i>-A\omega$ as currents. By Demailly's regularization theorem \cite{demailly1992regularization} (for our use, the particular case treated in \cite[Theorem 1]{blocki--kolodziej} would be enough), there exist $F_{i, \ell}\in C^{\infty}(X)$ such that $F_{i, \ell}\searrow f_i$ and $\ddbar F_{i, \ell}>-A\omega$. Consider a family of Monge--Amp\`{e}re equations:
\begin{align}\label{Family of MAs}
    \left\{
    \begin{array}{ll}
      \displaystyle(\omega+\ddbar \phi_{i, \ell})^n=\frac{e^{f-F_{i, \ell}+\phi_{i, \ell}}\omega^n}{|s_i|_{h_i}^{2(1-\beta_i)}},  \\
     \omega_{i, \ell}=\omega+\ddbar \phi_{i, \ell}>0.
    \end{array}
    \right.
\end{align}
$\omega_{i, \ell}$ can be seen as a sequence of \KE\;edge metrics to approximate the \KE\;crossing edge metric $\omega_{\phi_\beta}$ for each $i$. 
\begin{lem}
$||\phi_{i, \ell}-\phi_\beta||_{C^0(X)}\to 0$ as $\ell\to \infty$ for each $i=1,\dots, r$.
\end{lem}
\begin{proof}
First, note that for each $i$, there exists a solution $\phi_{i, \ell} \in C^{2, \alpha_{i}, \beta_i}(X)$ to \eqref{Family of MAs} for some $\alpha_{i}\in (0, 1)$, where the H\"{o}lder space $C^{2, \alpha_{i}, \beta_i}(X)$ is defined as in \cite{donaldson2012kahler}. By integrating both sides of \eqref{MA eq} and \eqref{Family of MAs}, one finds $\phi_{i, \ell}$ pointwise converges to $\phi_\beta$ almost everywhere. By assumptions on $\beta_i$ and $\phi_{i, \ell}$, $\omega_{i, \ell}^n/\omega^n\in L^{p}(X, \omega)$ for some $p>1$. Thus $\omega_{i, \ell}^n/\omega^n \to \omega_{\phi_\beta}^n/\omega^n$ in $L^1(X, \omega)$. Then we finish the proof by applying Kolodziej's stability result for Monge--Amp\`{e}re equation \cite[Theorem 4.1]{kolodziej2003monge}.
\end{proof}

\textbf{Step 2: Comparing $\boldsymbol{\phi_{i, \ell}}$ with $\boldsymbol{\psi_{\beta_i}}$.}

We first compare the potential functions of $\omega_{i, \ell}$ and $\Omega_{\beta_i}$. Let $\tilde{\phi}_{i, \ell}:=\phi_{i, \ell}-\psi_{\beta_i}$, then using \eqref{Defi of Omegabetai} and \eqref{Family of MAs} we get
\begin{align}
    \displaystyle(\omega+\sqrt{-1}\partial \bar{\partial} \phi_{i, \ell})^n&=\frac{e^{f-F_{i, \ell}+\phi_{i, \ell}}\omega^n}{|s_i|_{h_i}^{2(1-\beta_i)}},
    \\
    \displaystyle\Rightarrow 
    (\omega+\sqrt{-1}\partial \bar{\partial} \psi_{\beta_i}+\tilde{\phi}_{i, \ell})^n&=\frac{e^{f-F_{i, \ell}+\psi_{\beta_i}+\tilde{\phi}_{i, \ell}}\omega^n}{|s_i|_{h_i}^{2(1-\beta_i)}},
    \\
    \Rightarrow
    (\Omega_{\beta_i}+\ddbar \tilde{\phi}_{i, \ell})^n&=e^{\tilde{\phi}_{i, \ell}+\tilde{F}_{i, \ell}}\Omega^n_{\beta_i},\label{MA for Omegabetai}
\end{align}
where $\displaystyle \tilde{F}_{i, \ell}=\psi_{\beta_i}-F_{i, \ell}+f+\log \left(\frac{\omega^n}{|s_i|_{h_i}^{2(1-\beta_i)}\Omega^n_{\beta_i}} \right)$. Then we claim that 
\begin{claim}
For some uniform constant $C>0$,
\begin{align}\label{Bd for Fil}
||\tilde{F}_{i, \ell}||_{C^0(X)}\leq C.
\end{align}
\end{claim}
\begin{proof}
First note that $F_{i, \ell}$ and $f$ are smooth on $X$ by construction, hence $F_{i, \ell}$ and $f$ are bounded as $X$ is compact. Therefore, it suffices to show
\begin{align*}
\tilde{F}_{i, \ell}+F_{i, \ell}-f&=\psi_{\beta_i}+\log \left(\frac{\omega^n}{|s_i|^{2(1-\beta_i)}_{h_i}\Omega_{\beta_i}^n} \right)\\
&=\log \left(\frac{\beta_i^2 \omega^n}{|s_i|_{h_i}^{2(1-\beta_i)}(1-|s_i|^{2\beta_i}_{h_i})^2 \Omega_{\beta_i}^n}\right)
\end{align*}
is bounded. Below we drop the $i$ in the subscript for simplicity. To prove the claim, it is equivalent to show
\begin{align*}
    \Omega_\beta^n =\frac{\beta^2}{|s|_h^{2(1-\beta)}(1-|s|^{2\beta}_h)^2} e^{O(1)} \omega^n
\end{align*}
near the divisor. Let $p\in M\setminus D$ near $D$. Let $e$ be a local holomorphic frame for $L_D$, and $(z_1,\dots, z_n)$ be a local holomorphic coordinate chart such that $s=z_1 e$. Let $h=e^{-\phi}$ be a smooth hermitian metric on $\mathcal{O}_X(D)$ and $\theta$ the curvature form of $(L_D, h)$. Denote
\begin{align*}
    \omega&=\sqrt{-1}g_{i\bar{j}} dz_i\wedge d\bar{z}_j,\\
    \theta&=\sqrt{-1}\phi_{i\bar{j}} dz_i\wedge d\bar{z}_j.
\end{align*}
Recall the expression \eqref{formula for omegabeta} of $\Omega_\beta$. We calculate
\begin{align*}
    \langle D^{1, 0}s, D^{1, 0}s\rangle&=e^{-\phi}(dz_1+z_1\frac{\partial \phi}{\partial z_k}dz_k)\wedge (d\bar{z}_1+\bar{z}_1\frac{\partial \phi}{\partial \bar{z}_k}d\bar{z}_k)\\
    &=e^{-\phi}[
    (1+\bar{z}_1\frac{\partial \phi}{\partial \bar{z}_k}+z_1\frac{\partial \phi}{\partial z_1}+|z_1|^2\frac{\partial \phi}{\partial z_1}\frac{\partial \phi}{\partial \bar{z}_1})dz_1\wedge d\bar{z}_1\\
    &+\sum_{k=2}^n (\bar{z}_1\frac{\partial \phi}{\partial \bar{z}_k}+|z_1|^2\frac{\partial \phi}{\partial z_1}\frac{\partial \phi}{\partial \bar{z}_k})dz_1\wedge d\bar{z}_k\\
    &+\sum_{k=2}^n ({z}_1\frac{\partial \phi}{\partial {z}_k}+|z_1|^2\frac{\partial \phi}{\partial z_k}\frac{\partial \phi}{\partial \bar{z}_1})dz_k\wedge d\bar{z}_1\\
    &+\sum_{k, l=2}^n |z_1|^2\frac{\partial \phi}{\partial z_k}\frac{\partial \phi}{\partial \bar{z}_l}dz_k\wedge d\bar{z}_l
    ].
\end{align*}
Hence, 
\begin{align*}
    &\frac{\beta^2}{|s|_h^{2(1-\beta)}(1-|s|_h^{2\beta})^2}\langle D^{1, 0}s, D^{1, 0}s\rangle\\
    &=(\frac{\beta^2}{|s|_h^{2(1-\beta)}(1-|s|_h^{2\beta})^2}+O(1))dz_1\wedge d\bar{z}_1+\sum_{k=2}^n (\frac{\beta^2|s|_h^{2\beta}}{z_1(1-|s|_h^{2\beta})^2}+O(1))dz_1\wedge d\bar{z}_k\\
    &+\sum_{k=2}^n (\frac{\beta^2|s|^{2\beta}_h}{\bar{z}_1(1-|s|_h^{2\beta})^2}+O(1))dz_k\wedge d\bar{z}_1+\sum_{k, l=2}^n\frac{\beta^2 |s|_h^{2\beta}}{(1-|s|_h^{2\beta})^2} dz_k\wedge d\bar{z}_l\\
    &=(\frac{\beta^2}{|s|_h^{2(1-\beta)}(1-|s|_h^{2\beta})^2}+O(1))dz_1\wedge d\bar{z}_1+\sum_{k=2}^n (\frac{\beta^2|s|_h^{2\beta}}{z_1(1-|s|_h^{2\beta})^2}+O(1))dz_1\wedge d\bar{z}_k\\
    &+\sum_{k=2}^n (\frac{\beta^2|s|^{2\beta}_h}{\bar{z}_1(1-|s|_h^{2\beta})^2}+O(1))dz_k\wedge d\bar{z}_1+\sum_{k, l=2}^n O(1) dz_k\wedge d\bar{z}_l.
\end{align*}
Let 
\begin{align*}
    (A_{ij})_{i, j=1}^n=(g_{i\bar{j}})_{i, j=1}^n-\frac{\beta |s|^{2\beta}_h}{1-|s|_h^{2\beta}}(\phi_{i\bar{j}})_{i, j=1}^n+\frac{\beta^2}{|s|_h^{2(1-\beta)}(1-|s|_h^{2\beta})^2}\langle D^{1, 0}s, D^{1, 0}s \rangle. 
\end{align*}
Write $(A_{ij})_{i, j=1}^n$ as a block matrix
\begin{equation}\label{bmat for A}
   (A_{ij})_{i, j=1}^n= \left[
    \begin{matrix}
    A_{11} & \vec{A_{r}}\\
    \vec{A_c} & (A_{ij})_{i, j=2}^n
    \end{matrix}
    \right],
\end{equation}
then
\begin{align*}
    A_{11}&=g_{1\bar{1}}-\frac{\beta |s|_h^{2\beta}}{1-|s|_h^{2\beta}}\phi_{1\bar{1}}+\frac{\beta^2}{|s|_h^{2(1-\beta)}(1-|s|_h^{2\beta})^2}+O(|s|_h^{2\beta-1})+O(|s|_h^{2\beta})O(1),\\
    A_{1j}&=g_{1\bar{j}}-\frac{\beta |s|_h^{2\beta}}{1-|s|_h^{2\beta}}\phi_{1\bar{j}}+O(|s|_h^{2\beta-1})+O(|s|_h^{2\beta})+O(1),\quad j=2,\dots, n,\\
    A_{i1}&=g_{i\bar{1}}-\frac{\beta |s|_h^{2\beta}}{1-|s|_h^{2\beta}}\phi_{i\bar{1}}+O(|s|_h^{2\beta -1})+O(|s|_h^{2\beta})+O(1),\quad i=2,\dots, n,\\
    A_{kl}&=O(|s|_h^{2\beta})+O(1),\quad k, l=2,\dots, n.
\end{align*}
Recall the formula for determinant of block matrices as in \eqref{bmat for A},
\begin{align}\label{det bmat}
    \det (A_{ij})_{i, j=1}^n=\det (A_{ij})_{i, j=2}^n\cdot (A_{11}-\vec{A_r}((A_{ij})_{i, j=2}^n)^{-1}\vec{A_c}).
\end{align}
Using \eqref{det bmat}, 
\begin{align*}
    \frac{\Omega_\beta^n}{\omega^n}&=\frac{\det(g_{i\bar{j}}-\frac{\beta |s|_h^{2\beta}}{1-|s|_h^{2\beta}}\phi_{i\bar{j}}+\frac{\beta^2}{|s|_h^{2(1-\beta)}(1-|s|_h^{2\beta})^2}\langle D^{1, 0}s, D^{1, 0}s \rangle)}{\det (g_{i\bar{j}})}\\
    &=\frac{\det (A_{ij})_{i, j=1}^n}{\det (g_{i\bar{j}})}\\
    &=e^{O(1)}\cdot \det (A_{ij})_{i, j=2}^n\cdot (A_{11}-\vec{A_r}((A_{ij})_{i, j=2}^n)^{-1}\vec{A_c})\\
    &=e^{O(1)}\cdot \left(\frac{\beta^2}{|s|_h^{2(1-\beta)}(1-|s|_h^{2\beta})^2}+O(|s|_h^{4\beta-2})+O(|s|_h^{2\beta-1})+O(|s|_h^{4\beta-1})+O(|s|_h^{2\beta})+O(|s|_h^{4\beta})+O(1)\right).
% \det \left[
%  \begin{matrix}
%   O(1)+\frac{\beta^2}{|s|_h^{2(1-\beta)}(1-|s|_h^{2\beta})^2} & O(1)+\frac{\beta^2|s|^{2\beta}_h}{z_1(1-|s|_h^{2\beta})^2} & \cdots & O(1)+\frac{\beta^2|s|^{2\beta}_h}{z_1(1-|s|_h^{2\beta})^2}\\
%   O(1)+\frac{\beta^2|s|^{2\beta}_h}{\bar{z}_1(1-|s|_h^{2\beta})^2} & O(1) & \cdots & O(1) \\
%   \vdots & \vdots & \vdots& \vdots\\
%   O(1)+\frac{\beta^2|s|^{2\beta}_h}{\bar{z}_1(1-|s|_h^{2\beta})^2}& O(1)& \cdots &O(1)
%   \end{matrix}
%   \right] 
\end{align*}
Thus, one finds that the dominant term is $\displaystyle\frac{\beta^2}{|s|_h^{2(1-\beta)}(1-|s|_h^{2\beta})^2}$. In another word, we have shown $\displaystyle\Omega_\beta^n=e^{O(1)}\frac{\beta^2}{|s|_h^{2(1-\beta)}(1-|s|_h^{2\beta})^2} \omega^n$, which is exactly what we need.
\end{proof}

\begin{lem}\label{bound for tilde phi}
There exists a uniform constant $C>0$ such that for any $\ell$,
\begin{align*}
    \sup_{X\setminus D_i} |\tilde {\phi}_{i, \ell}|\leq C.
\end{align*}
\end{lem}
\begin{proof}
Let $\chi_{i, \ell, \epsilon}=\tilde{\phi}_{i, \ell}+\epsilon \log |s_i|^2_{h_i}$ for small $\epsilon>0$. Since $\chi_{i, \ell, \epsilon}(p)$ approaches $-\infty$ when $p\to D_i$,   $\chi_{i, \ell, \epsilon}$ obtains its maximum on $X\setminus D_i$, at say $p_{\max}$. Then
\begin{align*}
    0\geq \ddbar \tilde{\phi}_{i, \ell}(p_{\max} )-\epsilon \theta_i(p_{\max}),
\end{align*}
where $\theta_i$ is the curvature of the Chern connection on $(L_{D_i}, h_i)$.
Then at $p_{\max}$, 
\begin{align}
    (\Omega_{\beta_i}+\ddbar \tilde{\phi}_{i, \ell})^n &\leq (\Omega_{\beta_i}+\epsilon \theta_i)^n\\
    &\leq 2^n \Omega^n_{\beta_i},
    \label{Est for Omegabetai}
\end{align}
by the fact that $\Omega_{\beta_i}\geq \epsilon \theta_i$ for small enough $\epsilon$, as shown in Lemma \ref{shouchong}.
Combining \eqref{MA for Omegabetai} and \eqref{Est for Omegabetai}, at $p_{\max}$, 
\begin{align*}
    e^{\tilde{\phi}_{i, \ell}+\tilde{F}_{i, \ell}}(p_{\max})&\leq 2^n\\
    \Rightarrow \tilde{\phi}_{i, \ell}(p_{\max})&\leq n\log 2-\tilde{F}_{i, \ell}(p_{\max})\\
    &\leq n\log 2-\inf_{X\setminus D_i} \tilde{F}_{i, \ell}.
\end{align*}
For any $p\in X\setminus D_i$, 
\begin{align*}
    \tilde{\phi}_{i, \ell}(p)&=\chi_{i, \ell, \epsilon}(p)-\epsilon \log |s_i|^2_{h_i}(p)\\
    &\leq \chi_{i, \ell, \epsilon}(p_{\max})-\epsilon \log |s_i|^2_{h_i}(p)\\
    &\leq n\log 2-\inf_{X\setminus D_i}\tilde{F}_{i, \ell}+\epsilon \log |s_i|^2_{h_i}(p_{\max})-\epsilon \log |s_i|_{h_i}^2 (p)\\
    &\leq C
\end{align*}
for some constant $C>0$, when letting $\epsilon\to 0$ and using \eqref{Bd for Fil}.
Similarly by considering $\tilde{\chi}_{i, \ell, \epsilon}:=\tilde{\phi}_{i, \ell}-\epsilon \log |s_i|^2_{h_i}$ achieving its minimum on $X\setminus D_i$, we can show a lower bound for $\tilde{\phi}_{i, \ell}$ on $X\setminus D_i$.
% \begin{lem}\label{bound for tilde phi}
% There exists a uniform constant $C>0$ such that for any $\ell$, when $\beta_i$ is small enough,
% \begin{align*}
%     \sup_{X\setminus D_i} \tilde {\phi}_{i, \ell}\leq C.
% \end{align*}
% \end{lem}
\end{proof}

\textbf{Step 3: The Laplacian estimates for $\boldsymbol{\omega_{i, \ell}}$ and $\boldsymbol{\Omega_{\beta_i}}$.}

In this section, we use Chern--Lu's inequality to deduce the lower bound on $\omega_{i, \ell}$ with respect to $\Omega_{\beta_i}$.
\begin{lem}\label{compare omegail with omegabetai}
Assume $\beta\in(0, \frac{1}{2}]^r$. There exists a uniform constant $C>0$, indepent of $\beta_i$, such that for any $i=1,\dots, r$ and all $\ell$, 
\begin{align*}
    \omega_{i, \ell}\geq C\Omega_{\beta_i},
\end{align*}
on $X\setminus \operatorname{supp}(D_i)$.
\end{lem}
\begin{proof}
Consider the identity map
\begin{align*}
    \operatorname{id}: (X\setminus D_i, \omega_{i, \ell}=\omega+\ddbar \phi_{i, \ell})\to (X\setminus D_i, \Omega_{\beta_i}).
\end{align*}
Recall that $\ddbar F_{i, \ell}>-A\omega$ and $\omega\leq 2r\Omega_{\beta_i}$ (from Lemma \ref{shouchong}), thus $\ddbar F_{i, \ell}> -2 A r \Omega_{\beta_i}$. Thus, by \eqref{Family of MAs}, $\operatorname{Ric}\omega_{i, \ell}>-\omega_{i, \ell}-C_2 \Omega_{\beta_i}$ for some constant $C_2>0$. From Lemma \ref{erchong}, $|\operatorname{Bisec}_{\Omega_{\beta_i}}|\leq C_3$ for some constant $C_3>0$ when $\beta_i\in(0, \frac{1}{2}]$.
Then by Chern--Lu's inequality \cite[Proposition 7.1]{JMR} (see also \cite[Proposition 7.2]{rubinstein2014smooth}), 
\begin{align}\label{est for laplacian of tromega}
    \Delta_{\omega_{i, \ell}}(\log \operatorname{tr}_{\omega_{i, \ell}}\Omega_{\beta_i}-(C_2+2C_3+1)\tilde{\phi}_{i, \ell})\geq -1-(C_2+2C_3+1)n+\operatorname{tr}_{\omega_{i, \ell}}\Omega_{\beta_i}.
\end{align}
Set for $0<\epsilon\ll 1$,
\begin{align*}
    H_{i, \ell, \epsilon}=\log \operatorname{tr}_{\omega_{i, \ell}}\Omega_{\beta_i}-(C_2+2C_3+1)\tilde{\phi}_{i, \ell}+\epsilon \log |s_i|^2_{h_i},
\end{align*}
then 
\begin{align}
    \Delta_{\omega_{i, \ell}} H_{i, \ell, \epsilon}&=\Delta_{\omega_{i, \ell}}(\log \operatorname{tr}_{\omega_{i, \ell}}\Omega_{\beta_i}-(C_2+2C_3+1)\tilde{\phi}_{i, \ell})-\epsilon \operatorname{tr}_{\omega_{i, \ell}}\theta_i\\
    &\geq \Delta_{\omega_{i, \ell}} (\log \operatorname{tr}_{\omega_{i, \ell}}\Omega_{\beta_i}-(C_2+2C_3+1)\tilde{\phi}_{i, \ell})-\frac{1}{2}\operatorname{tr}_{\omega_{i, \ell}\Omega_{\beta_i}},\label{est for laplacian H}
\end{align}
where the last inequality is true by noting that  $\theta_i\leq M \Omega_{\beta_i}$ for some constant $M>0$ and assuming $\epsilon<\frac{1}{2M}$.

Combine \eqref{est for laplacian of tromega} and \eqref{est for laplacian H},
\begin{align}\label{better est for H}
    \Delta_{\omega_{i, \ell}} H_{i, \ell, \epsilon}\geq \frac{1}{2} \operatorname{tr}_{\omega_{i, \ell}}\Omega_{\beta_i}-C
\end{align}
for some $C>0$. 
$H_{i, \ell, \epsilon}$ achieves its maximum on $X\setminus D_i$, at $q_{\max}$. Then by \eqref{better est for H},
\begin{align*}
    \operatorname{tr}_{\omega_{i, \ell}}\Omega_{\beta_i}(q_{\max})\leq 2C.
\end{align*}
For any $q\in X\setminus D_i$,
\begin{align*}
    \log \operatorname{tr}_{\omega_{i, \ell}}\Omega_{\beta_i}(q)&=H_{i, \ell, \epsilon}(q)+(C_2+2C_3+1) \tilde{\phi}_{i, \ell}(q)-\epsilon \log |s_i|_{h_i}^2 (q)\\
    &\leq H_{i, \ell, \epsilon}(q_{\max})+(C_2+2C_3+1)\tilde{\phi}_{i, \ell}(q)-\epsilon\log |s_i|^2_{h_i}(q)\\
    &\leq 2C- (C_2+2C_3+1)\tilde{\phi}_{i, \ell}(q_{\max})+\epsilon\log |s_i|^2_{h_i}(q_{\max})\\
    &+(C_2+2C_3+1)\tilde{\phi}_{i, \ell}(q)-\epsilon \log |s_i|^2_{h_i}(q)\\
    &\leq \;\text{some constant}\; C,
\end{align*}
where the last inequality is true by Lemma \ref{bound for tilde phi} and letting $\epsilon \to 0$. Hence we have shown
\begin{align}
    \omega_{i, \ell}\geq C\cdot \Omega_{\beta_i},\label{ineq: omegail and omegabetai}
\end{align}
on $X\setminus \operatorname{supp}(D_i)$.
\end{proof}
It remains to show a lower bound for $\omega_{\phi_\beta}$ as in \eqref{ineq: omegail and omegabetai}. By \eqref{ineq: omegail and omegabetai} and the fact that $\omega_{i, \ell}^n$ are uniformly equivalent to $\Omega_{\beta_i}$, $\omega_{i, \ell}$ has uniformly bounded mass away from $\operatorname{supp}(D_i)$. Thus, the Evans--Krylov estimates and  the usual bootstrapping for complex Monge--Amp\`{e}re equations give all derivatives' estimates for $\omega_{i, \ell}$ away from the divisor. Proposition \ref{comp omegaphi and Omegabetai} together with the above discussion show that $\phi_{i, \ell}\to \phi_{\beta}$ in $C^\infty_{\operatorname{loc}}(X\setminus \operatorname{supp}(D_\beta))$. Hence by taking $\ell \to \infty$ in \eqref{ineq: omegail and omegabetai}, we get \eqref{omegaphi g Omegai} pointwise.
\end{proof}

\subsection{Global convergence of \texorpdfstring{$\boldsymbol{\omega_{\phi_\beta}}$}{KE metrics}}\label{sec: glo convergence of kee}

A smooth K\"{a}hler metric $\Omega_{PC}$ on $X\setminus D$ is said to have mixed cusp and edge singularities along a divisor $D$ if whenever $D$ is locally given by $D=\sum_{i=1}^t \{z_i=0\}+\sum_{j=t+1}^m (1-\beta_j) \{z_j=0\}$ with
$t<m\leq n$, $\Omega_{PC}$ is quasi-isometric to the following metric:
\begin{align*}
    \omega_{PC}:=\sum_{i=1}^t \frac{\sqrt{-1}dz_i\wedge d\bar{z}_i}{|z_i|^2\log^2 |z_i|^2}+\sum_{j=t+1}^m \frac{\beta_j^2 \sqrt{-1}dz_j\wedge d\bar{z}_j}{|z_j|^{2(1-\beta_j)}}+\sum_{\ell=m+1}^n \sqrt{-1} dz_\ell\wedge d\bar{z}_\ell.
\end{align*}
In particular, when $t=m$, $\omega_{PC}$ has merely cusp singularities along $D$.\\
In the case $t=m$, it is well known \cite{kobayashi1984kahler,TianYau} that if $K_X+D$ is ample, there exists a unique \KE\;metric on $X\setminus D$ with cusp singularities along $D$. In general, it is shown that if $K_X+D$ is ample, there exists a unique \KE\;metric on $X\setminus D$ with mixed cusp and cone singularities along $D$ \cite[Theorem A]{Gmix}.
As a corollary of Theorem \ref{compare omegaphi and Omegabeta}, we study the global weak convergence and local smooth convergence of the \KE\;crossing edge metrics $\omega_{\phi_\beta}$ to a \KE\;mixed cusp and edge metric on $(X, D_\beta)$ as some of the cone angles tend to $0$.
The first observation is the following lemma.
\begin{lem}\label{convergence of the refer metric}
Assume $\beta_i\to 0,$ for $i=1,\dots, t<r$, and $\beta_j\to d_j\in (0, 1)$ for $j=t+1,\dots, r$,
then $\Omega_\beta$ weakly converges to some \K\;mixed cusp and edge metric $\Omega_{PC}$. Moreover, $\Omega_{\beta}$ converges to $\Omega_{PC}$ in $C^\infty_{\operatorname{loc}}(X\setminus \operatorname{supp}(D_\beta))$.
\end{lem}
\begin{proof}
Recall the definition of $\Omega_\beta$. Note that $\displaystyle\log\left[({1-|s_i|_{h_i}^{2\beta_i})}/{\beta_i} \right]^2$ converges to $\log\log^2 |s_i|_{h_i}^2$ in $L^1(X, \omega)$ and $C^\infty_{\operatorname{loc}}(X\setminus \operatorname{supp}(D_\beta))$ as $\beta_i\to 0$ for each $i=1,\dots, t$. Thus $\Omega_\beta$ converges to 
\begin{align*}
    \Omega_{PC}:=\omega-\sum_{i=1}^t \ddbar \log\log^2|s_i|^2_{h_i}-\sum_{j=t+1}^r \ddbar \log\left[\frac{1-|s_j|_{h_j}^{2d_j}}{d_j} \right]^2
\end{align*}
in $C^\infty_{\operatorname{loc}}(X\setminus \operatorname{supp}(D_\beta))$ sense and weakly in the sense of currents. It remains to show that $\Omega_{PC}$ has mixed cusp and edge singularities along $D_\beta$. To see this, recall we denote by $\theta_i$ the Chern curvature form of $(L_{D_i}, h_i)$ for each $i$, then by \eqref{eq: forms related to s norm} and \eqref{eq: D10s and curv}, we calculate that
\begin{align*}
    &\sum_{i=1}^t \ddbar \log\log^2|s_i|^2_{h_i}\\
    =&\sum_{i=1}^t 2\sqrt{-1}\cdot \frac{(\partial\bar{\partial}|s_i|^2_{h_i})(\log |s_i|^2_{h_i})(|s_i|^2_{h_i})-\partial (\log|s_i|^2_{h_i}|s_i|^2_{h_i})\bar{\partial}|s_i|^2_{h_i}}{(\log|s_i|_{h_i}^2)^2(|s_i|_{h_i}^2)^2}\\
    =&\sum_{i=1}^r \frac{2\sqrt{-1}\langle D^{1, 0}s_i, D^{1, 0}s_i\rangle}{\log^2|s_i|^2_{h_i}|s_i|^2_{h_i}}+\frac{2}{\log|s_i|_{h_i}^2}\theta_i.
\end{align*}
Thus, $\Omega_{PC}$ has cusp singularities along $D_i$ for $i=1,\dots, t$. The result follows.
\end{proof}

\begin{thm}\label{rmk on weak convergence}
The \KE\;crossing edge metric $\omega_{\phi_\beta}$ converges to the \KE\;mixed cusp and edge metric on $(X, D_\beta)$ globally in a weak sense and locally in a strong sense when $\beta_i\to 0$ for $i=1,\dots, t<r$ and $\beta_j\to d_j\in (0, 1)$ for $j=t+1,\dots, r$.
\end{thm}

\begin{proof}
By Theorem \ref{compare omegaphi and Omegabeta}, the family of $\omega_{\phi_\beta}$ has uniformly bounded mass. Thus, the family of $\omega_{\phi_\beta}$ is relatively compact in the weak topology. The same arguments in the proof of Lemma \ref{bound for tilde phi} and in the end of the proof of Proposition \ref{comp omegaphi and Omegabetai}  give respectively the $C^0$-estimate and all higher-order estimates for the family of $\omega_{\phi_\beta}$. Therefore, any weak limit $\omega_0$ is smooth on $X\setminus \operatorname{supp}(D_\beta)$ and this $C^\infty_{\operatorname{loc}}$-convergence indicates that such $\omega_0$ is \KE\;outside $D_\beta$. Lemma \ref{convergence of the refer metric} shows any such $\omega_0$ also admits mixed cusp and edge singularities along $D_\beta$. Thus, by Yau's generalized maximum principle \cite{yau1975harmonic}, all such $\omega_0$ coincides with the unique \KE\;metric on $X\setminus \operatorname{supp}(D_\beta)$ with mixed cusp and cone singularities along $D_\beta$. Hence we have shown the locally strong and globally weak convergence of $\omega_{\phi_\beta}$ to a \KE\;mixed cusp and edge metric as $\beta_i\to 0$ for $i=1,\dots, t$ and $\beta_j\to d_j$ for $j=t+1,\dots, r$.
\end{proof}
%-------------------------------------------
\section{Asymptotic behavior near the divisors in the small angle limit}\label{sec: asm behavior near D}

Theorem \ref{rmk on weak convergence} only gives us the smooth convergence of $\omega_{\phi_\beta}$ to a \KE\;mixed cusp and edge metric away from the divisor when cone angles approach $0$. In this section, we study the asymptotic behavior of $\omega_{\phi_\beta}$ near $D$ when some of the cone angles tend to $0$. More precisely, consider a fixed point $p\in D_\beta$ with a holomorphic coordinate chart $(U, \{z_i\}_{i=1}^n)$ centered at $p$ such that $D_\beta \cap U=\{z_1\cdots z_m=0\}$, for $m\leq n$ and $D_j\cap U=\{z_j=0\}$ for $j=1,\dots, m$. Let $\beta_i$ denote the cone angle along $D_i$ for each $i$. From now on, assume $\beta_1\leq \beta_2\leq\cdots \leq \beta_m$. We allow other cone angles to tend to $0$, but we always assume that $\beta_1$ goes to $0$ in the fastest speed, i.e., $\beta_1/\beta_i \nrightarrow +\infty$, for $i=2,\dots, m$. 

\subsection{A small neighborhood of \texorpdfstring{$\boldsymbol{D_\beta}$}{D}}\label{sec: nbhd of d}
By choosing an appropriate coordinate system \cite[Lemma 4.1]{CGP}, whenever $D_\beta$ is locally given by $\{z_1 \cdots z_m=0\}$, the reference metric $\Omega_\beta$ is equivalent to the following metric:
\begin{align}\label{model metric with triv. metric}
    \omega_{\beta, \operatorname{mod}}:=\sum_{i=1}^m \frac{\beta_i^2 \sqrt{-1} dz_i\wedge d\bar{z}_i}{|z_i|^{2(1-\beta_i)}(1-|z_i|^{2\beta_i})^2}+\sum_{j=m+1}^n \sqrt{-1} dz_j\wedge d\bar{z}_j.
\end{align}
Thus, Theorem \ref{compare omegaphi and Omegabeta} tells us on $X\setminus \operatorname{supp}(D_\beta)$, there exists a uniform constant $C>0$ such that 
\begin{align}
    C^{-1}\omega_{\beta, \operatorname{mod}}\leq \omega_{\phi_\beta}\leq C\omega_{\beta, \operatorname{mod}}.\label{eq:equivalence of kee and model metric}
\end{align}
Thanks to \eqref{eq:equivalence of kee and model metric}, it is enough to consider $((\mathbb{C}^*)^m\times \mathbb{C}^{n-m}, \omega_{\beta, \operatorname{mod}})$ when dealing with a small neighborhood of $D_\beta$ under the metric $\omega_{\phi_\beta}$. Let us fix a point $p\in D_\beta$. Let $(U, z_1,\dots, z_n)$ be a holomorphic coordinate chart centered at $p$, such that $U\cap D_\beta=\{z_1\cdots z_m=0\}$ and $U\cap D_i=\{z_i=0\}$ for $i=1,\dots, m$. Let $\mathbb{D}:=\{|z_i|\leq 1, i=1,\dots, n\}$ be the unit polydisk. Then we claim that the distance function $d_\beta$ induced by the completion of $\omega_{\beta, \operatorname{mod}}$ on $\mathbb{D}$ satisfies
\begin{align}
    d_\beta(0, z)\simeq \sum_{i=1}^m \frac{1}{2}\log \left(\frac{1+|z_i|^{\beta_i}}{1-|z_i|^{\beta_i}} \right)+\sum_{j=m+1}^n |z_j|,\quad z\in\mathbb{D},\label{eq:distance function}
\end{align}
where "$\simeq$" means "is equivalent up to a constant independent of $z$ to". Indeed, $\displaystyle\frac{1}{2}\log \left(\frac{1+x^{\beta_i}}{1-x^{\beta_i}} \right)$ is the primitive of $\displaystyle\frac{\beta_i}{x^{1-\beta_i}(1-x^{2\beta_i})}$, and \eqref{eq:distance function} follows from this fact and \eqref{model metric with triv. metric}. Summarizing the discussions above, it is enough to study the polydisk in $\mathbb{C}^n$
\begin{align*}
    \left\{|z_i|^{\beta_i}< \frac{1-e^{-2a}}{1+e^{-2a}}, i=1,\dots, m, z_j<a, j=m+1, \dots, n\right\},\quad a>0,
\end{align*}
when we study a neighborhood of $D_\beta$ given by the geodesic ball $B_{\omega_{\phi_\beta}}(p, a)$ centerd at $p$ of radius $a$ with respect to the metric $\omega_{\phi_\beta}$.

\subsection{The mixed cylinder and edge metric}
In this section, we focus on a small neighborhood of $D_\beta$ and show that in a neighborhood of $D_\beta$, a renormalization of $\omega_{\beta, \operatorname{mod}}$ locally converges to a mixed cylinder and edge metric (see Definition \ref{def of mix} below) in the $C^\infty$-sense.

\begin{defi}\label{def of mix}
A K\"{a}hler metric $\tilde{\omega}$ on $(\mathbb{C}^*)^m\times \mathbb{C}^{n-m}$ is called a mixed cylinder and edge metric if $\tilde{\omega}$ is quasi-isometric to the following metric:
\begin{align*}
    \omega_{\operatorname{mix}}:=\sum_{i=1}^t \frac{\sqrt{-1} dz_i\wedge d\bar{z}_i}{|z_i|^2}+\sum_{j=t+1}^m \frac{\beta_j^2 \sqrt{-1} dz_j\wedge d\bar{z}_j}{|z_j|^{2(1-\beta_j)}}+\sum_{\ell=m+1}^n \sqrt{-1} dz_\ell\wedge d\bar{z}_{\ell},
\end{align*}
where $\beta_j\in(0, 1)$ for $j=t+1,\dots, m$.
\end{defi}

Denote by
\begin{align*}
    \mathbb{D}(a_1,\dots, a_m, b):=\{z\in (\mathbb{C}^*)^m\times \mathbb{C}^{n-m}: |z_i|<a_i, i=1,\dots, m, |z_j|<b, j=m+1,\dots, n\}.
\end{align*}
Let 
\begin{align*}
    U_{\beta}:=\mathbb{D}\left(e^{-\frac{1}{2\beta_1}}, \left(\frac{\beta_1}{\beta_2}\right)^{\frac{1}{2\beta_2}}, \dots, \left(\frac{\beta_1}{\beta_m}\right)^{\frac{1}{2\beta_m}}, 1\right).
\end{align*}
From section \ref{sec: nbhd of d}, one realizes $U_\beta$ as a neighborhood of $D_\beta$. We endow $U_\beta$ with $\displaystyle\frac{1}{\beta_1^2} \omega_{\beta, \operatorname{mod}}$. Define a map 
\begin{align*}
    \Psi_\beta: \mathbb{D}\left(e^{\frac{1}{2\beta_1}}, \left(\frac{\beta_2}{\beta_1}\right)^{\frac{1}{2\beta_2}}, \dots, \left(\frac{\beta_m}{\beta_1}\right)^{\frac{1}{2\beta_m}}, \frac{1}{\beta_1}\right)\to U_\beta=\mathbb{D}\left(e^{-\frac{1}{2\beta_1}}, \left(\frac{\beta_1}{\beta_2}\right)^{\frac{1}{2\beta_2}}, \dots, \left(\frac{\beta_1}{\beta_m}\right)^{\frac{1}{2\beta_m}}, 1\right)\\
    (w_1, \dots, w_m, w_{m+1}, \dots, w_n)\mapsto \left(e^{-\frac{1}{\beta_1}}w_1, \left(\frac{\beta_1}{\beta_2}\right)^{\frac{1}{\beta_2}}w_2,\dots, \left(\frac{\beta_1}{\beta_m}\right)^{\frac{1}{\beta_m}}w_m, \beta_1 w_{m+1}, \dots, \beta_1 w_n\right).
\end{align*}
On $\Psi_\beta^{-1}(U_\beta)$, the pull-back metric reads
\begin{align}
    \Psi_\beta^* (\frac{1}{\beta_1^2}\omegamod)&=\frac{e^{-2}|w_1|^{2\beta_1}}{(1-e^{-2}|w_1|^{2\beta_1})^2}\cdot \frac{\sqrt{-1} dw_1\wedge d\bar{w}_1}{|w_1|^2}+\sum_{i=2}^m \frac{\sqrt{-1} dw_i \wedge d\bar{w}_i}{|w_i|^{2(1-\beta_i)}(1-\frac{\beta_1^2}{\beta_i^2}|w_i|^{2\beta_i})^2}\\
    &+\sum_{j=m+1}^n \sqrt{-1} dw_j \wedge d\bar{w}_j.\label{eq: pullback of normalized model metric}
\end{align}
Note that for $(w_1, \dots, w_m, w_{m+1}, \dots, w_n)\in (\mathbb{C}^*)^m\times \mathbb{C}^{n-m}$, $|w_1|^{2\beta_1}\to 1$ as $\beta_1\to 0$ and $\frac{\beta_1^2}{\beta_i^2}|w_i|^{2\beta_i}\to 0$ as $\frac{\beta_1}{\beta_i}\to 0$. Moreover, for any compact subset $K\subset (\mathbb{C}^*)^m\times \mathbb{C}^{n-m}$, when $\beta_1$ is small enough, $K\subset \Psi_\beta^{-1}(U_\beta)$. Hence we have indeed shown the following result.

\begin{lem}\label{convergence of the pull-back metric}
The pull-back of $\frac{1}{\beta_1^2}\omega_{\beta, \operatorname{mod}}$ by $\Psi_\beta$ on any compact subset $K\subset (\mathbb{C}^*)^m\times \mathbb{C}^{n-m}$ converges to a mixed cylindrical and conical metric in $C^\infty(K)$ when $\beta_1\to 0$ and $\beta_i$ does not converge to $0$ for each $i=2,\dots, m$.
\end{lem}
\begin{proof}
Summarizing the discussions above, $\Psi_\beta^*(\frac{1}{\beta_1^2}\omegamod)$ converges to
\begin{align*}
    \frac{e^{-2}}{(1-e^{-2})^2}\cdot \frac{\sqrt{-1} d w_1 \wedge d\bar{w}_1}{|w_1|^2}+\sum_{i=2}^m \frac{\sqrt{-1} d w_i \wedge d\bar{w}_i}{|w_i|^{2(1-\beta_i)}}+\sum_{j=m+1}^n \sqrt{-1} dw_j \wedge d\bar{w}_j=: \hat{\omega},
\end{align*}
which is a mixed cylinder and edge metric by Definition \ref{def of mix}, in $C^\infty (K)$ as $\beta_1\to 0$ and $\frac{\beta_1}{\beta_i}\to 0, \forall i=2, \dots, m$.
\end{proof}

\subsection{The convergence of renormalized \texorpdfstring{$\boldsymbol{\omega_{\phi_\beta}}$}{KE metrics} near \texorpdfstring{$\boldsymbol{D_\beta}$}{D}}

For a K\"{a}hler metric $\xi$ on $\mathbb{C}^n$, let us denote $\displaystyle\bar \xi:=\Psi_\beta^*(\frac{1}{\beta_1^2}\xi)$.

\begin{thm}\label{thm: convergence of kee near divisor}
Let $\{\beta_{1, k}\}_{k\in\mathbb{N}}$ be a sequence of positive numbers converging to $0$. Assume further that $\lim_{k\to \infty}\beta_{i, k}>0$ for each $i = 2,\dots, r$. Assume all $\beta_{i, k}\in (0, \frac{1}{2}]$. Let ${\omega}_{\phi_{\beta_k}}$ be the (negatively curved) \KE\;crossing edge metric on $(X, D_k=\sum_{i=1}^r (1-\beta_{i, k})D_i)$.  Then there exists a subsequence of the metric spaces $\left(U_{\beta_k}, \frac{1}{\beta_{1, k}^2}\omega_{\phi_{\beta_k}}\right)$ which converges in pointed Gromov-Hausdorff topology to $((\mathbb{C}^*)^m\times \mathbb{C}^{n-m}, \bar{\omega}_\infty)$, where $\bar{\omega}_\infty$ is a mixed cylindrical and conical metric. Indeed, a subsequence of $\bar{\omega}_{\phi_{\beta_k}}$ converges in $C^\infty_{\operatorname{loc}}((\mathbb{C}^*)^m\times \mathbb{C}^{n-m})$-topology to $\bar{\omega}_\infty$.  
\end{thm}
\begin{proof}
First note that $\omega_{\phi_\beta}$ admits a potential function on $U_\beta$ since $\omega_{\beta, \operatorname{mod}}$ admits one. Thus, $\bar{\omega}_{\phi_\beta}$ admits a potential function on $\Psi_\beta^{-1}(U_\beta)$, denoted by $\bar{\phi}_\beta$. The proof consists of three steps. We first deduce the $C^0$-estimate of $\bar{\phi}_\beta$ using Theorem \ref{compare omegaphi and Omegabeta}. Then we derive the $C^{2, \alpha}$-estimates for $\bar{\phi}_\beta$ by the standard regularization arguments for Monge--Amp\`{e}re equations. This combining with Arzel\`{a}--Ascoli Theorem gives us a cluster value of the metrics. Finally, we use that smooth convergence to conclude the pointed Gromov--Hausdorff convergence as wanted.

\textbf{Step 1: $\boldsymbol{C^0}$-estimates.}

Discussions in section \ref{sec: nbhd of d} indicate that there exists a uniform constant $C>0$ (independent of $\beta\in(0, \frac{1}{2}]^r$) such that 
\begin{align*}
    C^{-1}\omegamod\leq \omega_{\phi_\beta}\leq C\omegamod
\end{align*}
on $U_\beta$.
Thus 
\begin{align*}
    C^{-1} \Psi_\beta^*(\frac{1}{\beta_1^2}\omegamod)\leq \bar{\omega}_{\phi_\beta}\leq C\Psi_\beta^*(\frac{1}{\beta_1^2}\omegamod).
\end{align*}
By Lemma \ref{convergence of the pull-back metric}, $\Psi_\beta^*(\frac{1}{\beta_1^2}\omegamod)$ converges in $C^\infty (K)$ to $\hat{\omega}$ for any compact $K\subset (\mathbb{C}^*)^m\times \mathbb{C}^{n-m}$. Hence, there exists a constant $C_K>0$ (independent of $\beta$) such that 
\begin{align}\label{comparison}
    C_K^{-1} \hat{\omega}\leq \bar{\omega}_{\phi_\beta}\leq C_K \hat{\omega}.
\end{align}
By \eqref{comparison}, $\bar{\omega}_{\phi_\beta}$ is uniformly bounded in mass on $K$. Then by the weak compactness of positive currents, $\bar{\phi}_\beta$ has a uniform $L^1_{\operatorname{loc}}$ bound, hence uniform $L^p_{\operatorname{loc}}$ bounds, for any $p>1$. By \eqref{comparison}, $\Delta \bar{\phi}_\beta$ is uniformly bounded. Thus by standard elliptic regularity results, \cite[Theorem 8.17]{GT}, there exists a constant $C$ independ of $\beta$ such that 
\begin{align}\label{bound for phibetabar}
    ||\bar{\phi}_\beta||_{C^0(K)}\leq C, \quad \text{for}\; C=C(K).
\end{align}

\textbf{Step 2: Higher-order estimates and the smooth local convergence.}

Since $\omega_{\phi_\beta}$ satisfies the \KE\;equation outside $D_\beta$, $\bar{\omega}_{\phi_\beta}$ satisfies
\begin{align}\label{MA for omegaphibar}
    \operatorname{Ric} \bar{\omega}_{\phi_\beta}=-\beta_1^2 \bar{\omega}_{\phi_\beta}\quad \text{on} \; K.
\end{align}
Let $dV_{\operatorname{eucl}}$ denote the Euclidean volume form on $(\mathbb{C}^*)^m\times \mathbb{C}^{n-m}$. Define 
\begin{align*}
    H_\beta:=\log \frac{\bar{\omega}_{\phi_\beta}^n e^{-\beta_1^2\bar{\phi}_\beta}}{dV_{\operatorname{eucl}}}.
\end{align*}
$H_\beta$ is pluriharmonic by \eqref{MA for omegaphibar}.  By the definition of $H_\beta$, 
\begin{align}\label{MA for phibetabar}
    (\ddbar \bar{\phi}_\beta)^n=e^{\beta_1^2 \bar{\phi}_\beta+H_\beta}dV_{\operatorname{eucl}}.
\end{align}
By \eqref{comparison}, 
\begin{align}\label{ineq: b1H}
    ||\beta_1^2 \bar{\phi}_\beta+ H_\beta||_{C^0(K)}<+\infty.
\end{align}
Combining \eqref{bound for phibetabar} and \eqref{ineq: b1H}, we see $||H_\beta||_{C^0(K)}<+\infty$. Then by gradient estimates for pluriharmonic functions, 
\begin{align}\label{grad est for H}
    ||H_\beta||_{C^k(K)}<C(k, K),\quad \text{where}\; C(k, K) \;\text{only depends on $k$ and $K$ not on $\beta$.}
\end{align}
Define
\begin{align*}
    \Phi: \phi\mapsto \log \frac{(\ddbar \phi)^n e^{-\beta_1^2\phi}}{dV_{\operatorname{eucl}}}.
\end{align*}
$\Phi$ is a uniform elliptic concave operator as a function of $\partial \bar{\partial}\phi$. Hence by Evans--Krylov theory, $||\phi||_{C^{2, \alpha}(K)}$ is controlled by $||\phi||_{C^0(K^\prime)}, ||\Delta \phi||_{C^0(K^\prime)}$ and $||\Phi(\phi)||_{C^{0, 1}(K^\prime)}$ for some $K^\prime \Supset K$. 
Since $\Phi(\bar{\phi}_\beta)=H_\beta$, by \eqref{grad est for H}, \eqref{bound for phibetabar} and the fact $||\Delta {\bar{\phi}_\beta}||_{C^0(K')}<+\infty$, there exist some $\alpha\in (0, 1)$ and a uniform constant $C>0$ such that 
\begin{align*}
    ||\bar{\phi}_\beta||_{C^{2, \alpha}(K)}\leq C.
\end{align*}
By standard bootstrapping arguments, every derivative of $\bar{\phi}_\beta$ is uniformly bounded on $K$. Then Arzel\`{a}-Ascoli theorem indicates $(\bar{\phi}_{\beta_k})_{k\in\mathbb{N}}$ has a convergent subsequence in $C^\infty(K)$-topology.
Recall \eqref{MA for omegaphibar}, letting $\beta_1\to 0$ then due to the $C^\infty$-convergence above we get a cluster value $\bar{\omega}_\infty$ such that 
\begin{align*}
\operatorname{Ric} \bar{\omega}_\infty=0.
\end{align*}
By \eqref{comparison}, $\bar{\omega}_\infty$ is quasi-isometric to $\hat{\omega}$, and therefore is a mixed cylinder and edge metric. 

\textbf{Step 3: Pointed Gromov--Hausdorff convergence}

It remains to show a subsequence of $(U_{\beta_k}, \frac{1}{\beta_{1, k}^2}\omega_{\phi_{\beta_k}})$ converges in pointed Gromov-Hausdorff topology to $((\mathbb{C}^*)^m\times \mathbb{C}^{n-m}, \bar{\omega}_\infty)$. Fix $q\in (\mathbb{C}^*)^m\times \mathbb{C}^{n-m}$ and fix a radius $a>0$. First note that by construction, $\displaystyle B_{\bar{\omega}_{\phi_{\beta_k}}}(q, a)$ is isometric to $\displaystyle B_{\frac{1}{\beta_{1, k}^2}\omega_{\phi_{\beta_k}}}(\Psi_{\beta_k}(q), a)$. Secondly, by letting the index $k\in\mathbb{N}$ be large enough, we have $B_{\bar{\omega}_\infty}(q, 2a)\subset \Psi_{\beta_k}^{-1}(U_{\beta_k})$. Finally, due to the local $C^\infty$-convergence, $B_{\bar{\omega}_{\phi_{\beta_k}}}(q, a)\subset B_{\bar{\omega}_\infty}(q, 2a)$ and $B_{\bar{\omega}_{\phi_{\beta_k}}}(q, a)$ converges to $B_{\bar{\omega}_\infty}(q, a)$ in the Gromov-Hausdorff topology. Therefore $(U_{\beta_k}, \frac{1}{\beta_{1, k}^2}\omega_{\phi_{\beta_k}})$ converges (up to a subsequence) in pointed Gromov-Hausdorff topology to $((\mathbb{C}^*)^m\times \mathbb{C}^{n-m}, \bar{\omega}_\infty)$ by \cite[Definition 8.1.1]{BBI}.
\end{proof}

If we further allow more than one cone angles converge to $0$, then we have the following results by modifying the result of Lemma \ref{convergence of the pull-back metric}.

\begin{thm}
Let $\{\beta_{1, k}\}_{k\in\mathbb{N}}$ be a sequence of positive numbers converging to $0$. Assume further that for any $i=2, \dots, r$ such that $\{\beta_{i, k}\}_{k\in\mathbb{N}}$ also converges to $0$, there holds $\lim_{k\to\infty}\frac{\beta_{1, k}}{\beta_{i, k}}\in[0, 1]$. Assume $\beta_{i, k}\in (0, \frac{1}{2}]$ for $i=1, 2,\dots, r$ and all $k\in\mathbb{N}$. Let ${\omega}_{\phi_{\beta_k}}$ be the (negatively curved) \KE\;crossing edge metric on $(X, D_k=\sum_{i=1}^r (1-\beta_{i, k})D_i)$. Then there exists a subsequence of the metric spaces $(U_{\beta_k}, \frac{1}{\beta_{1, k}^2}\omega_{\phi_{\beta_k}})$ converging in pointed Gromov-Hausdorff topology to $((\mathbb{C}^*)^m\times \mathbb{C}^{n-m}, \bar{\omega}_\infty)$, where $\bar{\omega}_\infty$ is a mixed cylinder and edge metric with cylindrical part along components whose cone angles converge to $0$ and conical part along other components.
\end{thm}

\begin{proof}
Without loss of generality, assume
\begin{align}\label{assumption for angles 1}
    \lim_{k\to \infty}\beta_{i, k}&=0,\quad \text{for}\;i=1,\dots, t, t<m,\\
    \lim_{k\to \infty}\beta_{j, k}&:=\beta_{j, \infty}>0,\quad \text{for}\;j=t+1,\dots, m.
\end{align}
Moreover, assume
\begin{align}\label{assumption for angles 2}
    \lim_{k\to \infty}\frac{\beta_{1, k}}{\beta_{\ell, k}}&=c_\ell\in(0, 1],\quad \text{for}\;\ell=1, \dots, s, s<t,\\
    \lim_{k\to \infty}\frac{\beta_{1, k}}{\beta_{\ell, k}}&=0,\quad \text{for}\;\ell=s+1, \dots, t.
\end{align}
Recall in Lemma \ref{convergence of the pull-back metric}, we denote 
\begin{align*}
    \mathbb{D}(a_1,\dots, a_m, b)=\{z\in(\mathbb{C}^*)^m\times \mathbb{C}^{n-m}:|z_i|<a_i, i=1,\dots, m, |z_j|<b, j=m+1,\dots, n\},
\end{align*}
and
\begin{align*}
    \Psi_\beta: \mathbb{D}\left(e^{\frac{1}{2\beta_1}}, \left(\frac{\beta_2}{\beta_1}\right)^{\frac{1}{2\beta_2}}, \dots, \left(\frac{\beta_m}{\beta_1}\right)^{\frac{1}{2\beta_m}}, \frac{1}{\beta_1}\right)\to U_\beta:=\mathbb{D}\left(e^{-\frac{1}{2\beta_1}}, \left(\frac{\beta_1}{\beta_2}\right)^{\frac{1}{2\beta_2}}, \dots, \left(\frac{\beta_1}{\beta_m}\right)^{\frac{1}{2\beta_m}}, 1\right)\\
    (w_1, \dots, w_m, w_{m+1}, \dots, w_n)\mapsto \left(e^{-\frac{1}{\beta_1}}w_1, \left(\frac{\beta_1}{\beta_2}\right)^{\frac{1}{\beta_2}}w_2,\dots, \left(\frac{\beta_1}{\beta_m}\right)^{\frac{1}{\beta_m}}w_m, \beta_1 w_{m+1}, \dots, \beta_1 w_n\right).
\end{align*}
Now let us modify $\Psi_\beta$ by defining
\begin{align*}
    V_\beta:=\mathbb{D}\left(e^{-\frac{1}{2\beta_1}}, e^{-\frac{1}{2\beta_2}},\dots, e^{-\frac{1}{2\beta_s}}, \left(\frac{\beta_1}{\beta_{s+1}}\right)^{\frac{1}{2\beta_{s+1}}}, \dots, \left(\frac{\beta_1}{\beta_m}\right)^{\frac{1}{2\beta_m}}, 1 \right),
\end{align*}
and
\begin{align*}
    &\Phi_\beta: \mathbb{D}\left(e^{\frac{1}{2\beta_1}}, e^{\frac{1}{2\beta_2}},\dots, e^{\frac{1}{2\beta_s}}, \left(\frac{\beta_{s+1}}{\beta_1}\right)^{\frac{1}{2\beta_{s+1}}}, \dots, \left(\frac{\beta_m}{\beta_1}\right)^{\frac{1}{2\beta_m}}, \frac{1}{\beta_1}\right)\to  V_\beta,\\
    &\Phi_\beta(w_1, \dots, w_s, w_{s+1},\dots, w_m, w_{m+1}, \dots, w_n)=\\ &\left(e^{-\frac{1}{\beta_1}}w_1, e^{-\frac{1}{\beta_2}}w_2,\dots, e^{-\frac{1}{\beta_s}}w_s,  \left(\frac{\beta_1}{\beta_{s+1}}\right)^{\frac{1}{\beta_{s+1}}}w_{s+1},\dots, \left(\frac{\beta_1}{\beta_m}\right)^{\frac{1}{\beta_m}}w_m, \beta_1 w_{m+1}, \dots, \beta_1 w_n\right).
\end{align*}
Then
\begin{align*}
    \Phi_\beta^*\left(\frac{1}{\beta_1^2}\omega_{\beta, \operatorname{mod}}\right)&=\sum_{i=1}^s \frac{\beta_1^2}{\beta_i^2}\frac{e^{-2}|w_i|^{2\beta_i}}{(1-e^{-2}|w_i|^{2\beta_i})^2}\cdot\frac{ \sqrt{-1}dw_i\wedge d\bar{w}_i}{|w_i|^2}\\
    &+\sum_{j=s+1}^m \frac{\sqrt{-1}dw_j\wedge d\bar{w}_j}{|w_j|^{2(1-\beta_j)}\left(1-\frac{\beta_1^2}{\beta_j^2}|w_j|^{2\beta_j}\right)^2}+\sum_{\ell=m+1}^n \sqrt{-1}dw_\ell\wedge d\bar{w}_\ell.
\end{align*}
Denote 
\begin{align*}
    \beta_{k}:=(\beta_{1, k},\dots, \beta_{r, k}), \quad \text{for}\; k\in\mathbb{N}^*.
\end{align*}
For a compact $K\subset (\mathbb{C}^*)^m\times \mathbb{C}^{n-m}$, there exists a large enough $k$ such that $K\subset \Phi_{\beta_k}^{-1}(V_{\beta_k})$. By assumptions \eqref{assumption for angles 1} and \eqref{assumption for angles 2}, $\displaystyle \Phi_{\beta_k}^*\left(\frac{1}{\beta_{1, k}^2}\omega_{\beta_k, \operatorname{mod}}\right)$ converges in $C^{\infty}(K)$ to 
\begin{align*}
    \sum_{i=1}^s & \frac{e^{-2}}{c_i^2 (1-e^{-2})}\frac{\sqrt{-1}dw_i\wedge d\bar{w}_i}{|w_i|^2}+\sum_{j=s+1}^t \frac{\sqrt{-1}dw_j\wedge d\bar{w}_j}{|w_j|^2}\\
    &+\sum_{\ell=t+1}^m \frac{\sqrt{-1}dw_\ell\wedge d\bar{w}_\ell}{|w_\ell|^{2(1-\beta_{\ell, \infty})}}+\sum_{q=m+1}^n\sqrt{-1}dw_q\wedge d\bar{w}_q=:\check{\omega},
\end{align*}
which is a mixed cylindrical and conical metric by Definition \ref{def of mix}. The cylindrical parts are along $D_i, i=1,\dots, t$, whose cone angle goes to $0$. The remainder of the proof is similar to that of Theorem \ref{thm: convergence of kee near divisor}.

\end{proof}

\bibliographystyle{siam}
\bibliography{newbib}

\providecommand{\MR}[1]{}
\begin{thebibliography}{10}

\bibitem{blocki--kolodziej}
{\sc Z.~B{\l}ocki and S.~Ko{\l}odziej}, {\em On regularization of
  plurisubharmonic functions on manifolds}, Proc. Amer. Math. Soc., 135 (2007),
  pp.~2089--2093.

\bibitem{BBI}
{\sc D.~Burago, Y.~Burago, and S.~Ivanov}, {\em A course in metric geometry},
  vol.~33 of Graduate Studies in Mathematics, American Mathematical Society,
  Providence, RI, 2001.

\bibitem{CGP}
{\sc F.~Campana, H.~Guenancia, and M.~P\u{a}un}, {\em Metrics with cone
  singularities along normal crossing divisors and holomorphic tensor fields},
  Ann. Sci. \'{E}c. Norm. Sup\'{e}r. (4), 46 (2013), pp.~879--916.

\bibitem{Cheltsov-Rubinstein}
{\sc I.~A. Cheltsov and Y.~A. Rubinstein}, {\em Asymptotically log {F}ano
  varieties}, Adv. Math., 285 (2015), pp.~1241--1300.

\bibitem{Datar-Song}
{\sc V.~V. Datar and J.~Song}, {\em A remark on {K}\"{a}hler metrics with
  conical singularities along a simple normal crossing divisor}, Bull. Lond.
  Math. Soc., 47 (2015), pp.~1010--1013.

\bibitem{demailly1992regularization}
{\sc J.-P. Demailly}, {\em Regularization of closed positive currents and
  intersection theory}, J. Algebraic Geom., 1 (1992), pp.~361--409.

\bibitem{donaldson2012kahler}
{\sc S.~K. Donaldson}, {\em K\"{a}hler metrics with cone singularities along a
  divisor}, in Essays in mathematics and its applications, Springer,
  Heidelberg, 2012, pp.~49--79.

\bibitem{GT}
{\sc D.~Gilbarg and N.~S. Trudinger}, {\em Elliptic partial differential
  equations of second order}, vol.~224 of Grundlehren der Mathematischen
  Wissenschaften [Fundamental Principles of Mathematical Sciences],
  Springer-Verlag, Berlin, second~ed., 1983.

\bibitem{GH}
{\sc P.~Griffiths and J.~Harris}, {\em Principles of algebraic geometry},
  Wiley-Interscience [John Wiley \& Sons], New York, 1978.
\newblock Pure and Applied Mathematics.

\bibitem{Gmix}
{\sc H.~Guenancia}, {\em K\"{a}hler-{E}instein metrics with mixed
  {P}oincar\'{e} and cone singularities along a normal crossing divisor}, Ann.
  Inst. Fourier (Grenoble), 64 (2014), pp.~1291--1330.

\bibitem{guenancia2020kahler}
\leavevmode\vrule height 2pt depth -1.6pt width 23pt, {\em
  K\"{a}hler-{E}instein metrics: from cones to cusps}, J. Reine Angew. Math.,
  759 (2020), pp.~1--27.

\bibitem{GP}
{\sc H.~Guenancia and M.~P\u{a}un}, {\em Conic singularities metrics with
  prescribed {R}icci curvature: general cone angles along normal crossing
  divisors}, J. Differential Geom., 103 (2016), pp.~15--57.

\bibitem{JMR}
{\sc T.~Jeffres, R.~Mazzeo, and Y.~A. Rubinstein}, {\em K\"{a}hler-{E}instein
  metrics with edge singularities}, Ann. of Math. (2), 183 (2016), pp.~95--176.

\bibitem{kobayashi1984kahler}
{\sc R.~Kobayashi}, {\em K\"{a}hler-{E}instein metric on an open algebraic
  manifold}, Osaka J. Math., 21 (1984), pp.~399--418.

\bibitem{kolodziej2003monge}
{\sc S.~Ko{\l}odziej}, {\em The {M}onge-{A}mp\`ere equation on compact
  {K}\"{a}hler manifolds}, Indiana Univ. Math. J., 52 (2003), pp.~667--686.

\bibitem{LS}
{\sc A.~Lin and L.~Shen}, {\em Conic {K}\"{a}hler-{E}instein metrics along
  simple normal crossing divisors on {F}ano manifolds}, J. Funct. Anal., 275
  (2018), pp.~300--328.

\bibitem{mazzeo1999kahler}
{\sc R.~Mazzeo}, {\em K\"{a}hler-{E}instein metrics singular along a smooth
  divisor}, in Journ\'{e}es ``\'{E}quations aux {D}\'{e}riv\'{e}es
  {P}artielles'' ({S}aint-{J}ean-de-{M}onts, 1999), Univ. Nantes, Nantes, 1999,
  pp.~Exp. No. VI, 10.

\bibitem{Mazzeo--Rubinstein}
{\sc R.~Mazzeo and Y.~A. Rubinstein}, {\em The {R}icci continuity method for
  the complex {M}onge-{A}mp\`ere equation, with applications to
  {K}\"{a}hler-{E}instein edge metrics}, C. R. Math. Acad. Sci. Paris, 350
  (2012), pp.~693--697.

\bibitem{rubinstein2014smooth}
{\sc Y.~A. Rubinstein}, {\em Smooth and singular {K}\"{a}hler-{E}instein
  metrics}, in Geometric and spectral analysis, vol.~630 of Contemp. Math.,
  Amer. Math. Soc., Providence, RI, 2014, pp.~45--138.

\bibitem{Yconvex}
\leavevmode\vrule height 2pt depth -1.6pt width 23pt, {\em High-dimensional
  convex sets arising in algebraic geometry}, in Geometric Aspects of
  Functional Analysis (B. Klartag, E. Milman, Eds.), vol.~2266 of Lecture Notes
  in Math., Springer, Cham, 2020, pp.~301--323.

\bibitem{RZ}
{\sc Y.~A. Rubinstein and K.~Zhang}, {\em Small angle limits of hamilton’s
  footballs}, Bull. Lond. Math. Soc., 52 (2020), pp.~189--199.

\bibitem{Tian96}
{\sc G.~Tian}, {\em K\"{a}hler-{E}instein metrics on algebraic manifolds}, in
  Transcendental methods in algebraic geometry ({C}etraro, 1994), vol.~1646 of
  Lecture Notes in Math., Springer, Berlin, 1996, pp.~143--185.

\bibitem{TianYau}
{\sc G.~Tian and S.-T. Yau}, {\em Existence of {K}\"{a}hler-{E}instein metrics
  on complete {K}\"{a}hler manifolds and their applications to algebraic
  geometry}, in Mathematical aspects of string theory ({S}an {D}iego, {C}alif.,
  1986), vol.~1 of Adv. Ser. Math. Phys., World Sci. Publishing, Singapore,
  1987, pp.~574--628.

\bibitem{wu2006higher}
{\sc D.~Wu}, {\em K\"{a}hler-{E}instein metrics of negative {R}icci curvature
  on general quasi-projective manifolds}, Comm. Anal. Geom., 16 (2008),
  pp.~395--435.

\bibitem{yau1975harmonic}
{\sc S.~T. Yau}, {\em Harmonic functions on complete {R}iemannian manifolds},
  Comm. Pure Appl. Math., 28 (1975), pp.~201--228.

\end{thebibliography}

\end{document}